\documentclass[12pt]{article}

\usepackage[T1]{fontenc}
\usepackage[active]{srcltx}
\usepackage{lmodern}
\usepackage{amsmath,amstext,amssymb,graphicx,graphics,color}
\usepackage{tikz}
\usepackage{mleftright}
\usepackage{pgfplots}
\usepackage{theorem}
\usepackage{cite}               
\usepackage{hyperref}           
\usepackage{bbm}                
\usepackage{fancybox}           
\usepackage[section]{placeins}  
\usepackage{subcaption}
\theorembodyfont{\normalfont\slshape} 
\newtheorem{definition}{Definition}
\newtheorem{theorem}{Theorem}
\newtheorem{proposition}{Proposition}[section]
\newtheorem{corollary}[proposition]{Corollary}

\newtheorem{lemma}[proposition]{Lemma}
\theoremstyle{break} 

\newenvironment{proof}%
{{\par\noindent \bf Proof. \nobreak}}%
{\nobreak \removelastskip \nobreak \hfill $\Box$ \medbreak}
{{\par\noindent \bf Proof \nobreak}}%
{\nobreak \removelastskip \nobreak \hfill $\Box$ \medbreak}
{{\par\noindent \bf Proof lemma. \nobreak}}%
{\nobreak \removelastskip \nobreak \bf End proof lemma. \medbreak}
\newenvironment{remark}{\par \medskip \noindent {\bf Remark. }\nobreak}{\par \medskip}
\usepackage{fancyhdr}
 \fancypagestyle{plain}{
  \fancyhead{}
  
}

\def\paragraph#1{{\bf #1\ }}
\newcommand{\expo}{\mathrm{e}}

\newcommand{\dd}{\mathrm{d}}

\newcommand{\R}{{\mathbb{R}}}
\newcommand{\eps}{\varepsilon}
\newcommand{\emp}{\mathrm{emp}}
\newcommand{\DD}{\mathrm{D}}
\newcommand{\KL}{\mathrm{KL}}

\newcommand{\HH}{\mathrm{H}}
\newcommand{\MM}{\mathrm{M}}
\newcommand{\overbar}[1]{\mkern 1.5mu\overline{\mkern-1.5mu#1\mkern-1.5mu}\mkern 1.5mu}

\usepackage[utf8]{inputenc}
\usepackage{unicode_character}

\def\Proof{\noindent{\bf Proof}\quad}
\def\qed{\hfill$\square$\smallskip}

\title{From interacting agents to Boltzmann-Gibbs distribution of money}

\author{Pierre-Emmanuel Jabin \footnotemark[1] \and Fei Cao\footnotemark[2]}

\begin{document}
\maketitle

\footnotetext[1]{Pennsylvania State University - Department of Mathematics and Huck Institutes, State College, PA 16802, USA}
\footnotetext[2]{University of Massachusetts Amherst - Department of Mathematics and Statistics, Amherst, MA 01003, USA}

\tableofcontents

\begin{abstract}
We investigate the unbiased model for money exchanges: agents give at random time a dollar to one another (if they have one). Surprisingly, this dynamics eventually leads to a geometric distribution of wealth (shown empirically by Dragulescu and Yakovenko in \cite{dragulescu_statistical_2000} and rigorously in \cite{cao_derivation_2021,graham_rate_2009,lanchier_rigorous_2017,merle_cutoff_2019}). We prove a uniform-in-time propagation of chaos result as the number of agents goes to infinity, which links the stochastic dynamics to a deterministic infinite system of ordinary differential equations. This deterministic description is then analyzed by taking advantage of several entropy-entropy dissipation inequalities and we provide a quantitative almost-exponential rate of convergence toward the equilibrium (geometric distribution) in relative entropy.
\end{abstract}

\noindent {\bf Key words: Econophysics, Agent-based model, Propagation of chaos, Relative entropy, Logarithmic Sobolev inequalities}

\section{Introduction}
\setcounter{equation}{0}

\subsection{Background and main results}
In this work, we consider a simple mechanism for money exchange in a closed economical system, meaning that there are a fixed number of agents, denoted by $N$, with an (fixed) average number of dollars $\mu \in \mathbb{N}_+$. We denote by $S_i(t)$ the amount of dollars the agent $i$ has at time $t$. Since it is a closed economical system, we have:
\begin{equation}\label{eq:preserved_sum}
S_1(t)+ \cdots +S_N(t) = N\mu = \textrm{Constant} \qquad \text{for all } t\geq 0.
\end{equation}
Specifically, we consider the model proposed in \cite{dragulescu_statistical_2000}: at random time (generated by an exponential law), an agent $i$ is picked uniformly at random and if it has one dollar (i.e. $S_i\geq 1$) it will give it to another agent $j$ picked uniformly at random. If $i$ does not have one dollar (i.e. $S_i= 0$), then nothing happens. This model is termed as the \textbf{unbiased exchange model} in \cite{cao_derivation_2021} (as all the agents are being picked with equal probability) and can be represented by \eqref{unbiased_exchange}.
\begin{equation}
\label{unbiased_exchange}
\textbf{unbiased exchange:} \qquad   (S_i,S_j)~ \begin{tikzpicture} \draw [->,decorate,decoration={snake,amplitude=.4mm,segment length=2mm,post length=1mm}]
(0,0) -- (.6,0); \node[above,red] at (0.3,0) {\small{$λ$}};\end{tikzpicture}~  (S_i-1,S_j+1) \quad (\text{if } S_i\geq 1).
\end{equation}
In other words, any agents with at least one dollar gives to all of the others agents at a fixed rate. In order to have the correct asymptotic as $N→+∞$ (so that the rate of a typical agent giving a dollar per unit time is of order $1$), we need to adjust the rate $λ$ (more exactly $λ\mathbbm{1}_{[1\!,+∞)}(S_i)$) by normalizing by $N$.

\begin{figure}[!htb]
  \centering
  \includegraphics[scale = 0.7]{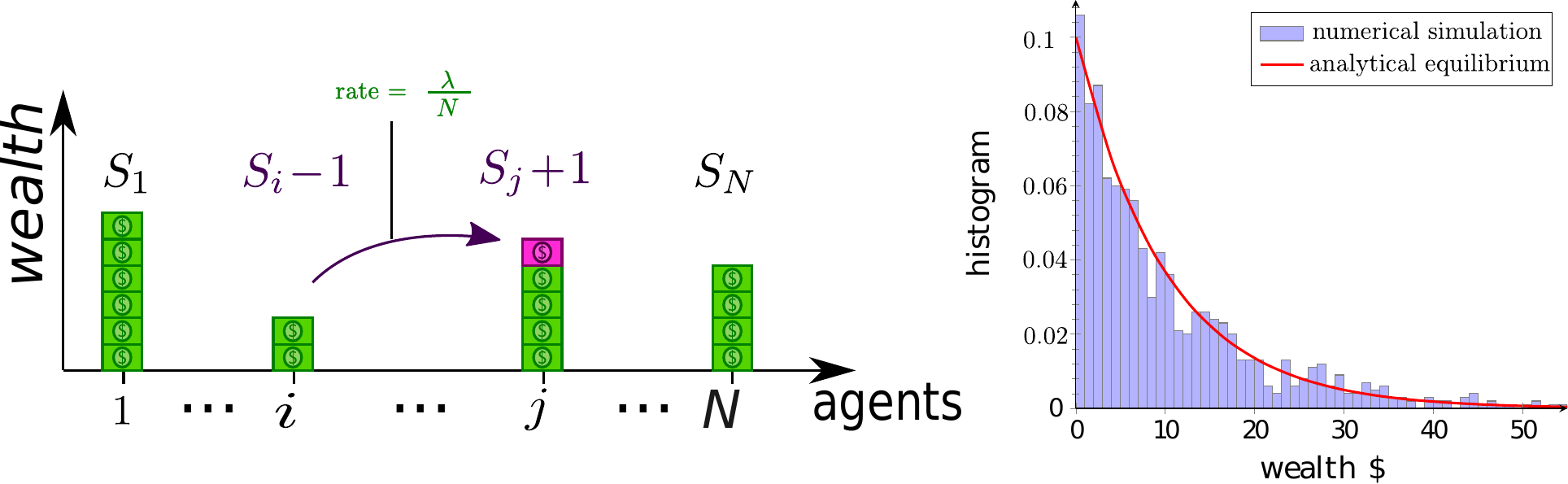}
  \caption{{\bf Left:} Illustration of the unbiased exchange model: at random time, one dollar is passed from a ``giver'' $i$ to a ``receiver'' $j$ at a fixed rate as long as the ``giver'' $i$ has at least a dollar. {\bf Right:} The distribution of wealth for the unbiased dynamics after 50, 000 steps with the average wealth per agent $\mu = 10$, this distribution is well-approximated by a exponential distribution with mean value $\mu = 10$.}
  \label{fig:illustration_model}
\end{figure}

We illustrate the dynamics in figure \ref{fig:illustration_model}-left. The fundamental question of interest is the exploration of the limiting money distribution among the agents as the total number of agents and time become large. We illustrate numerically in figure \ref{fig:illustration_model}-right the unbiased exchange dynamics using $N=1000$ agents. Notice that the wealth distribution is approximately exponential with the proportion of agents decaying as wealth increases.

If we denote by ${\bf p}(t)=\left(p_0(t),p_1(t),\ldots\right)$ the law of the process ${S}_1(t)$ as $N \to \infty$, i.e. $p_n(t) = \lim\limits_{N \to \infty} \mathbb P\left(S_1(t) = n\right)$, its time evolution is given by:
\begin{equation}
  \label{eq:law_limit}
  \frac{\dd}{\dd t} {\bf p}(t) = λ \,\mathcal{L}[{\bf p}(t)]
\end{equation}
with:
\begin{equation}
  \label{eq:L}
  \mathcal{L}[{\bf p}]_n:= \left\{
    \begin{array}{ll}
      p_1-\overbar{r}\,p_0 & \quad \text{if } n=0 \\
      p_{n+1}+\overbar{r}\,p_{n-1}- (1+\overbar{r})p_n & \quad \text{for } n \geq 1
    \end{array}
  \right.
\end{equation}
and $\overbar{r}=1-p_0$. Throughout this manuscript, we will put $\lambda = 1$ without any loss of generality. The link between the stochastic $N$-agents dynamics \eqref{unbiased_exchange} and the infinite system of ODEs \eqref{eq:law_limit} as $N \to \infty$ is refereed to as \emph{propagation of chaos} \cite{sznitman_topics_1991} and has been rigorously justified in this setting, see for instance \cite{cao_derivation_2021,graham_rate_2009,merle_cutoff_2019}. To be more precise, the work of \cite{cao_derivation_2021,graham_rate_2009,merle_cutoff_2019} only prove the validity of the evolution equation \eqref{eq:law_limit} for $t \in [0,T]$ with $T$ being arbitrary but fixed. In this work, we prove a uniform-in-time propagation of chaos for this mean-field dynamics, which refines earlier short time result on the propagation of chaos phenomenon.


\begin{remark}
The equation \eqref{eq:law_limit} can be seen as a jump process with loss and gain:
  \begin{itemize}
  \item \underline{gain} at $S_1$: two sources, the `$S_1+1$' that give a dollar and the `$S_1-1$' that receive a dollar ($S_1$ has to be greater than or equal to $1$ in that case)
    \begin{equation*}
      S_1+1 \to S_1 \quad \text{and} \quad S_1-1  \to S_1 \;\;\text{if}\quad S_1 \geq 1.
    \end{equation*}
  \item \underline{loss} at $S_1$, the `$S_1$' that give or receive a dollar
    \begin{equation*}
      S_1 \to S_1-1 \;\; \text{if } S_1\geq 1 \quad \text{and} \quad S_1 \to S_1+1.
    \end{equation*}
  \end{itemize}
\end{remark}

\begin{figure}[!htb]
\centering
\includegraphics[scale=1.0]{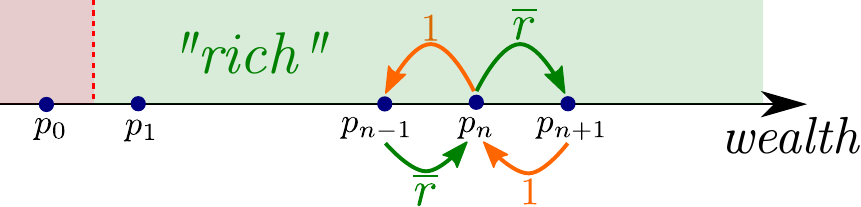}
\caption{Schematic illustration of the limiting ODE system \eqref{eq:law_limit} with $\lambda :=1$.}
\end{figure}

After we achieved the transition from the interacting system of agents \eqref{unbiased_exchange} to the deterministic system of nonlinear ODEs \eqref{eq:law_limit}, the natural follow-up step is to investigate the system \eqref{eq:law_limit} with the intention of proving convergence of solution of \eqref{eq:law_limit} to its (unique) equilibrium solution, see figure \ref{fig:scheme_sketch} for the general philosophy of the strategy, which has also been employed in \cite{cao_derivation_2021,graham_rate_2009,merle_cutoff_2019}.

\begin{figure}[!htb]
\centering
\includegraphics[scale = 0.7]{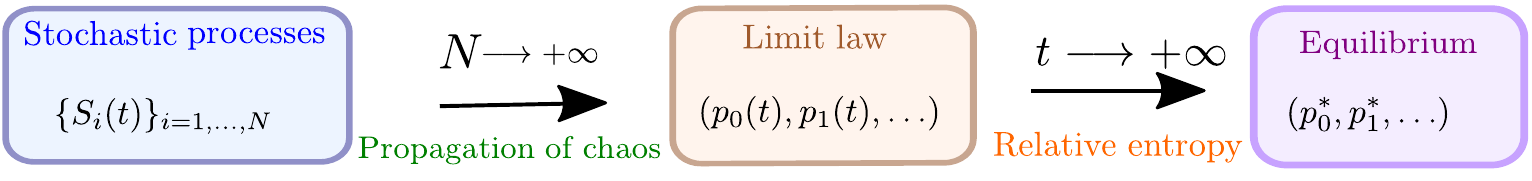}
\caption{Schematic illustration of the strategy of the limiting procedures.}
\label{fig:scheme_sketch}

\end{figure}

We quickly sketch some elementary observations concerning the ODE system \eqref{eq:law_limit} and refer interested readers to \cite{cao_derivation_2021,graham_rate_2009,merle_cutoff_2019} for detailed proofs and discussions.

\begin{lemma}\label{invariant}
If~${\bf p}(t)=\{p_n(t)\}_{n\geq 0}$ is the (unique) solution of \eqref{eq:law_limit}, then
\begin{equation}
\label{eq:conservation_mass_mean_value}
\sum_{n=0}^\infty \mathcal{L}[{\bf p}]_n =0,~~\textrm{and}~~ \sum_{n=0}^\infty n\,\mathcal{L}[{\bf p}]_n =0.
\end{equation}
In particular, we have ${\bf p}(t) \in V_\mu$ for all $t\geq 0$, where \[V_\mu:=\biggl\{\,{\bf p} \mathrel{\Big|} p_n \geq 0,~\sum_{n=0}^\infty p_n =1,~\sum_{n=0}^\infty n\,p_n =\mu \,\biggr\}\] is the space of probability mass functions on $\mathbb{N}$ (denoted by $\mathcal{P}(\mathbb N)$) with the prescribed mean value $\mu$. Moreover, the (unique) equilibrium distribution ${\bf p}^*=\{p^*_n\}_n$ in $V_\mu$ associated with \eqref{eq:law_limit} is given by:
\begin{equation}
\label{eqn:equil_limit}
p^*_n = p^*_0(r^*)^n,\quad n\geq 0,
\end{equation}
where $p^*_0 = 1 - r^* = \frac{1}{1+\mu}$ if we put initially that $\sum_{n=0}^\infty n\,p_n(0)=\mu$ for some $\mu \in \mathbb N_+$.
\end{lemma}

\begin{remark}
If $\mu = 0$, then the geometric equilibrium distribution \eqref{eqn:equil_limit} reduces to $p^*_0 = 1$ and $p^*_n = 0$ for all $n\geq 1$. Indeed, since nobody can give (as there are no ``rich'' people), the distribution ${\bf p}$ is not evolving. If the average number of dollars per agent $\mu$ is large, i.e., if $\mu \gg 1$, the unique equilibrium distribution ${\bf p}^*$ is well-approximated by a Boltzmann-Gibbs (exponential) distribution with mean value $\mu$, as \[p^*_n = \tfrac{1}{\mu}\,\expo^{-(n+1)\,\log(1+\frac{1}{\mu})} \approx \tfrac{1}{\mu}\,\expo^{-\frac{n}{\mu}}, \quad \forall\,n\geq 0 .\]
\end{remark}

Next, we recall the definition of entropy and relative entropy \cite{cover_elements_1999}, which will play a major role in the analysis of the large time behavior of the system \eqref{eq:law_limit}.

\begin{definition}
For a given probability mass function ${\bf p} \in V_\mu$, the entropy of ${\bf p}$ is defined via
\begin{equation}\label{entropy}
\HH[{\bf p}]=\sum_n p_n \log p_n.
\end{equation}
The relative entropy (or Kullback-Leiber divergence) from $\mathbf{p}^*$ to ${\bf p}$ is given by
\begin{equation}\label{relative_entropy}
\DD_{\KL} \left({\bf p}~||~{\bf p}^* \right)= \sum_{n=0}^\infty p_n\,\log \frac{p_n}{p^*_n}.
\end{equation}
\end{definition}

To prove convergence of the solution of \eqref{eq:law_limit} to its equilibrium solution \eqref{eqn:equil_limit}, a key observation is to realize that the original ODE dynamics \eqref{eq:law_limit} can be reformulated as a variant of a Fokker-Planck equation. Indeed, one has the following integration by parts formula:
\begin{lemma}
Let $\{p_n(t)\}_{n\geq 0}$ be the solution to \eqref{eq:law_limit} and $\varphi \colon \mathbb R \to \mathbb R$ to be a continuous function, then
\begin{equation}\label{IBP}
\sum_{n=0}^\infty p'_n\,\varphi(n)=\sum_{n=0}^\infty (\overbar{r}p_n-p_{n+1})\big(\varphi(n+1)-\varphi(n)\big).
\end{equation}
\end{lemma}

\noindent In particular, inserting $\varphi(n)=\log p_n$ yields the entropy dissipation result.

\begin{lemma}[\textbf{Entropy dissipation}]
Let ${\bf p}(t)=\{p_n(t)\}_{n\geq 0}$ be the solution to \eqref{eq:law_limit} and $\HH[{\bf p}]$ be the associated entropy, then for all $t>0$,
\begin{equation}\label{eq:entropy_dissipation}
\begin{aligned}
\frac{\dd}{\dd t}\HH[{\bf p}] &=  -\sum_{n=0}^\infty (\tilde{p}_{n+1} - p_{n+1})\left(\log(\tilde{p}_{n+1}) - \log(p_{n+1})\right) \\
&= -\mathrm{D}_{\mathrm{KL}} \left( {\bf p}~||~\mathbf{\tilde{p}} \right) - \DD_{\KL} \left({\bf \tilde{p}}~||~{\bf p} \right) \leq 0,
\end{aligned}
\end{equation}
where ${\bf \tilde{p}}:=\{\tilde{p}_n\}_{n\geq 0}$ is defined by $\tilde{p}_0=p_0$ and $\tilde{p}_n=\overbar{r}p_{n-1}$ for $n \geq 1$.
\end{lemma}

\begin{remark}
If we notice that $\frac{\dd}{\dd t} \sum_{n=0}^\infty p_n\log p^*_n = 0$ for all $t \geq 0$ along the flow of the solution to \eqref{eq:law_limit}, we can rewrite the entropy dissipation \eqref{eq:entropy_dissipation} as
\begin{equation}\label{eq:entropy_dissipation_rewrite}
\frac{\dd}{\dd t} \DD_{\KL}({\bf p}~||~{\bf p}^*) = -\mathrm{D}_{\mathrm{KL}} \left( {\bf p}~||~{\bf \tilde{p}} \right) - \DD_{\KL} \left({\bf \tilde{p}}~||~{\bf p} \right) \leq 0.
\end{equation}
\end{remark}

\noindent It has been shown in \cite{cao_derivation_2021,graham_rate_2009,merle_cutoff_2019} that we have ${\bf p}(t) \xrightarrow{t \to \infty} {\bf p}^*$ for each ${\bf p}(0) \in V_\mu$. Moreover, if the initial configuration of the stochastic system is $S_i(0) = \mu$ for all $1\leq i\leq N$ (which implies that the components of ${\bf p}(0)$ are all zero except a $1$ at its $(\mu+1)$-th spot), Theorem 1.2 in \cite{graham_rate_2009} showed that
\begin{equation}\label{eq:previous_result}
\limsup\limits_{t\to \infty} t^{-1}\log \DD_{\KL}\left({\bf p}(t)~||~{\bf p}^*\right) = -\Omega(\mu^{-2}),
\end{equation}
in which the notation $\Omega(\mu^{-2})$ is interpreted as follows: Given positive functions $f,g \colon \R_+ \to \R_+$, we say that $f = \Omega(g)$ if $f \geq c\,g$ for a constant $c > 0$. Unfortunately (at least from the perspective of the authors of this work), both \cite{graham_rate_2009} and \cite{merle_cutoff_2019} rely heavily on the particle system to get various a priori estimates that will be eventually leveraged to analyze the ODE system \eqref{eq:law_limit} (where everything is deterministic), so in the end their entropy dissipation analysis is \emph{not} independent of the investigation of the underlying stochastic particle systems \eqref{unbiased_exchange}. One of the main contributions of the present work is the derivation a entropy-entropy dissipation inequality (for general datum), relying solely on the nonlinear ODE system \eqref{eq:law_limit} itself, so that the ODE analysis part can be completely separated from the stochastic analysis part. In this regard, our approach will be accessible and beneficial to broader audiences. In particular, we will prove in Theorem \ref{thm2} that
\begin{equation}\label{eq:main_result_1}
\DD_{\KL}({\bf p}(t)~||~{\bf p}^*) \leq C\expo^{-C\sqrt{t}} \quad \forall\,t \geq 0
\end{equation}
holds for general initial datum ${\bf p}(0)$, where $C > 0$ is some fixed constant.

Although we will only investigate a specific binary exchange models in the present work, other exchange rules can also be imposed and studied, leading to different models. To name a few, the so-called immediate exchange model introduced in \cite{heinsalu_kinetic_2014} assumes that pairs of agents are randomly and uniformly picked at each random time, and each of the agents transfer a random fraction of its money to the other agents, where these fractions are independent and uniformly distributed in $[0,1]$. The so-called uniform reshuffling model investigated in \cite{dragulescu_statistical_2000,lanchier_rigorous_2018} suggests that the total amount of money of two randomly and uniformly picked agents possess before interaction is uniformly redistributed among the two agents after interaction. For closely related variants of the unbiased exchange model we refer to the recent work \cite{cao_derivation_2021}. For other models arising from econophysics, see \cite{chakraborti_statistical_2000, chatterjee_pareto_2004, lanchier_rigorous_2018-1,cao_explicit_2021}.

To the best of our knowledge, uniform in time propagation of chaos for other models coming from econophysics has only been studied in \cite{cortez_quantitative_2016,cortez_particle_2018} which include the uniform reshuffling model and the immediate exchange model as special cases. The pillars of the argument employed in \cite{cortez_quantitative_2016,cortez_particle_2018} is an ``optimal-coupling'' type argument, which can be dated back to 1970s \cite{murata_propagation_1977,tanaka_probabilistic_1978} and relies on a stochastic differential equation representation of the particle system in terms of Poisson random measures as well as a tricky application of the existence of a optimal coupling between the empirical measure and the law of the ``limit object''. Our approach to uniform in time propagation of chaos for the unbiased exchange model at hand will be based on a careful study of the entropy-entropy dissipation relation, at the level of the $N$-agent system as well as its limiting ODE system. In particular, our main theorem is stated as follows:
\begin{theorem}\label{thm4}
Let ${\bf q} = (q_0,q_1,\ldots)$ to be the individual coordinates of the empirical measure of the unbiased exchange dynamics, i.e., $q_n = \#\{i\in \mathbb N \mid S_i = n\} / N$ for all $n \in \mathbb N$. For all $t > 0$ and all $N \geq 1$, we have that
\begin{equation}\label{eq:L1chaos}
\mathbb{E}\,\|{\bf q}(t) - {\bf p}(t)\|^2_{\ell^1} \leq \frac{C}{1+\log N},
\end{equation}
in which ${\bf p}(t)=\{p_n(t)\}_{n\geq 0}$ is the unique solution of ODE system \eqref{eq:law_limit}. Furthermore, we have the entropic uniform in time propagation of chaos, i.e.,
\begin{equation}\label{eq:entropic_chaos}
\mathbb{E}\,\sum_n q_n(t)\,\log\frac{q_n(t)}{p_n(t)}\leq C\,\frac{\log \log N}{1+\log N}.
\end{equation}
\end{theorem}

\subsection{Ideas behind the proof of main results}

We plan to indicate the main steps and ideas used for the proof of our main theorems. We begin with the procedures towards the proof of the quantitative decay in relative entropy stated in \eqref{eq:main_result_1} (see also Theorem \ref{thm} at the beginning of section \ref{sec:2} below). We first establish a uniform in time upper bound for $\bar r = 1-p_0$ in Proposition \ref{prop1}, which enables us to deduce in Corollary \ref{cor1} the finiteness of the exponential moments of ${\bf p}(t)$ for all $t \geq 0$ whenever the initial datum ${\bf p}(0)$ has a finite exponential moment. The next key observation to make is that the entropy dissipation term serves as a upper bound for $|p_0 - p^*_0|$ (or for $|\bar{r} - r^*|$, which is the content of Lemma \ref{lem1}. The aforementioned crucial observation then allows us to derive a first entropy-entropy dissipation estimate, giving rise to a polynomial in time decay of the relative entropy (see Theorem \ref{thm1}):
\begin{equation}\label{eq:polynoimal_decay_intro}
\DD_{\KL}({\bf p}(t)~||~{\bf p}^*) \leq \frac{C}{t + C}.
\end{equation}
This entropy-entropy dissipation estimate will be refined (at large times) to obtain the improved upper bound \eqref{eq:main_result_1} stated in Theorem \ref{thm}. Such improvement relies crucially on the validity of certain logarithmic Sobolev inequalities, which will be justified by employing a technical theorem provided in \cite{canizo_trend_2017}.

We now turn to the ideas behind the proof of the uniform in time propagation of chaos (Theorem \ref{thm4}). The proof will be based on upper bounds on two relative entropies $\mathbb{E}\,\sum_{n\geq 0} q_n(t)\,\log \frac{q_n(t)}{p_n^*}$ and $\sum_{n\geq 0} p_n(t)\,\log \frac{p_n(t)}{p_n^*}$, the classical Csisz\'{a}r-Kullback-Pinsker inequality and a elementary interpolation argument. We will first establish in section \ref{subsec:3.1} a stochastic analog of the entropy-entropy dissipation relation in the deterministic setting, then in section \ref{subsec:3.2} we extend Corollary \ref{cor1} to the $N$-agent based dynamics, which is more involved than the deterministic setting as we need a relatively sharp large derivation estimate for the empirical proportion of ``rich'' people (i.e., for the random variable $r = 1 - q_0$). Section \ref{subsec:3.3} contains a precise quantitative logarithmic Sobolev inequality required for the proof of the uniform propagation of chaos result in section \ref{subsec:3.4}. Lastly, We believe it will be possible to modify our argument presented in this work so that uniform in time propagation of chaos result can be obtained for closely related models, for instance the so-called poor-biased and rich-biased exchange model investigated in \cite{cao_derivation_2021}.

\section{Quantitative entropy-entropy dissipation estimates}\label{sec:2}
\setcounter{equation}{0}

In this section, starting from the entropy dissipation relation
\begin{equation}\label{eq:ED_restate}
\frac{\dd}{\dd t} \HH = -\DD
\end{equation}
with $\HH := \DD_{\KL}({\bf p}~||~{\bf p}^*)$ and $\DD :=  \sum_{n=0}^\infty (\tilde{p}_{n+1} - p_{n+1})\left(\log(\tilde{p}_{n+1}) - \log(p_{n+1})\right)$, we derive a general entropy-entropy dissipation inequality based on certain logarithmic Sobolev inequalities (holding for a generic initial datum ${\bf p}(0) \in V_\mu$), which enables us to obtain a quantitative result on the decay of the relative entropy $\DD_{\KL}({\bf p}~||{\bf p}^*)$. In particular, the main theorem we will prove in this section reads as
\begin{theorem}\label{thm}
Assume that ${\bf p} \in V_\mu$ is the unique solution to the nonlinear system of ODEs \eqref{eq:law_limit}. There exists some $C > 0$ such that
\begin{equation}\label{eq:almost_exponential_decay_beginning}
\HH = \DD_{\KL}({\bf p}(t)~||~{\bf p}^*) \leq C\expo^{-C\sqrt{t}}.
\end{equation}
\end{theorem}
We emphasize that our analysis of the nonlinear ODE system \eqref{eq:law_limit} is purely analytic and fully independent of the underlying stochastic system, in contrast to the probabilistic approaches taken in \cite{graham_rate_2009,merle_cutoff_2019}.

\subsection{A uniform in time upper bound on $\bar{r}(t)$}\label{subsec:2.1}

In order to obtain the existence of some exponential moments of the solution ${\bf p}$ to \eqref{eq:law_limit}, we need an upper bound on $\bar{r} = \sum_{n\geq 1} p_n$ for any $t > 0$.

\begin{proposition}\label{prop1}
Assume that ${\bf p} \in V_\mu$ is the unique solution to the nonlinear system of ODEs \eqref{eq:law_limit}. There exists a fixed constant $C \in (0,1)$ depending only on $\mu$, such that $\bar{r}(t) < C$ for all $t > 0$.
\end{proposition}

\Proof The proof can be divided into three steps. The first step consists of proving the existence of a positive fixed integer $N_0 \in \mathbb{N}$ and a fixed constant $C_\mu > 1$ (depending only on $\mu$) such that for any $t \geq 0$, there exists some $n_0 \leq N_0$ such that $p_{n_0} \geq 1/C_\mu$. Indeed, for any integer $N > \mu$, we have \[\mu = \sum_{n\geq 1} np_n \geq N\sum_{n\geq N} p_n \geq N\left(1-N\max\limits_{k\geq N-1} p_k\right),\] whence $\max\limits_{k\leq N-1} p_k \geq (N-\mu)/N^2$. In the second step, we show that $\bar{r} \leq \phi(D)$ for each $t > 0$, where $\phi \colon [0,\infty) \to [0,\infty)$ is a non-decreasing function such that $\lim_{x \to \infty} \phi(x) = 1$. For this purpose, we observe that for any $a,b \in \mathbb{R}_+$, if $a \geq Mb$ for some $M \geq 2$, then there exists some constant $C > 0$ independent of $a,b$ and $M$, for which $(a-b)\log(a/b) \geq \frac{1}{C}\,a\log(a/b) \geq \frac{1}{C}\,a\log M$. With this elementary observation in mind, we deduce from the definition of the dissipation term $\DD$ that there exists some fixed constant $C > 0$ such that for all $n \in \mathbb{N}_+$ and $M \geq 2$,
\begin{equation}\label{eq:Basic}
p_{n+1}\mathbbm{1}_{\{Mp_n \leq p_{n+1}\}} \leq \frac{CD}{\log M}.
\end{equation}
Since there exists some $n_0 \leq N_0$ such that $p_{n_0} \geq 1/C_\mu$, if we choose $M_0$ large enough so that $1/C_\mu > \frac{CD}{\log M}$ (or equivalently $M_0 \approx \expo^{CC_\mu D}$), then $p_{n_0 -1} > \frac{p_{n_0}}{M_0} \geq frac{1}{C_\mu M_0}$. We can choose $M_1$ such that $\frac{1}{C_\mu M_0} > \frac{CD}{\log M_1}$ (or equivalently $M_1 \approx \expo^{CC_\mu M_0 D}$), then $p_{n_0 -2} > \frac{p_{n_1}}{M_1} \geq \frac{1}{C_\mu M_0 M_1}$. Iterating this process eventually leads us to \[p_0 = 1 - \bar{r} \geq 1/\exp\{CC_\mu D\exp\{CC_\mu D\exp\{\cdots \exp\{CC_\mu D\}\cdots\} \} \},\] where the exponential function appears at most $N_0$ times. This estimate completes the proof of our second step. For any $i \in \mathbb N$, integrating \label{eq:ED} from $t = i$ to $t = i + 1$ gives rise to $\int_i^{i+1} \DD(s) \dd s \leq C_0$, where $C_0 := \DD_{\KL}({\bf p}(0)~||~{\bf p}^*)$ is a fixed constant. Hence there exists some $t_i \in [i,i+1]$ such that $\DD(t_i) \leq C_0$, and by the previous analysis we deduce that $1 - \bar{r}(t_i) \geq 1 - \phi(C_0)$. Now for any $t \geq 1$ (the constant $1$ can be replaced by any positive constant by a obvious modification of the argument presented here), there exists some $i \in \mathbb N$ for which $t \in [i+1,i+2]$ so that $t_i \leq t \leq t_i + 2$. Finally, thanks to the differential inequality $\frac{\dd}{\dd t} (1 - \bar{r}) = p_1 - \bar{r}p_0 \geq -(1 - \bar{r})$, we conclude that \[1 - \bar{r}(t) \geq \left(1 - \bar{r}(t_i)\right)/\expo^{2} \geq \left(1-\phi(C_0)\right)/\expo^{2},\] which completes the proof. \qed

\begin{remark}
The essence of the content of Proposition \ref{prop1} coincides with Lemma 3 in \cite{merle_cutoff_2019}. However, while the argument provided in \cite{merle_cutoff_2019} is entirely probabilistic and is built upon a non-trivial explicit stochastic representation of the underlying particle system, our argument is solely analytic and perhaps more accessible to a wider audience.
\end{remark}

As a very important consequence of Proposition \ref{prop1}, we deduce the existence of some exponential moments of the solution ${\bf p}$ to \eqref{eq:law_limit}.

\begin{corollary}\label{cor1}
Assume that ${\bf p} \in V_\mu$ is the unique solution to the nonlinear system of ODEs \eqref{eq:law_limit}. There exists some fixed $K > 1$ such that \[\sup\limits_{t > 0} \sum_{n \geq 0} K^n p_n(t) < \infty \] if $\sum_{n \geq 0} K^n p_n(0)$ is finite.
\end{corollary}

\Proof It suffices to show the existence of a fixed constant $K > 1$ such that $\sup\limits_{t > 0} \sum_{n \geq 1} K^n p_n(t) < \infty$ when $\sum_{n \geq 1} K^n p_n(0) < \infty$. Let $\MM (t) :=  \sum_{n \geq 0} K^n p_n(t)$, a direct computation yields \[\frac{\dd}{\dd t} \MM = \frac{(\bar{r}K - 1)(K-1)}{K}\MM + \bar{r}Kp_0 - p_1.\] Thanks to Proposition \ref{prop1}, there exists a fixed constant $C \in (0,1)$ such that $\bar{r}(t) \leq C$ for all $t > 0$. To finish the proof, it suffices to take $K > 1$ such that $K \leq 1/C$. \qed

\subsection{The first entropy-entropy dissipation estimate}\label{subsec:2.2}

To derive our first entropy-entropy dissipation estimate, we first show that the entropy dissipation $\DD$ serves us a upper bound on $|p_0 - p^*_0| = |\bar{r} - r^*|$.

\begin{lemma}\label{lem1}
Assume that ${\bf p} \in V_\mu$ is the unique solution to the nonlinear system of ODEs \eqref{eq:law_limit}. Then there exists a fixed constant $C > 0$ such that
\begin{equation}\label{eq:pillar}
|p_0 - p^*_0| = |\bar{r} - r^*| \leq C\sqrt{\DD}
\end{equation}
for any $t > 0$.
\end{lemma}

\Proof We observe that for each $n \in \mathbb N$, we have
\begin{equation}\label{eq:identity}
p_n - \bar{r}^n p_0 = \sum_{0\leq k\leq n-1} \bar{r}^{n-k-1}\left(p_{k+1} - \bar{r}p_k\right),
\end{equation}
with the convention that sums over the empty set are equal to zero. Therefore, for some fixed constants $C, \tilde{C} > 0$,
\begin{align*}
\frac{\left|r^* - \bar{r}\right|}{(1-\bar{r})(1+\mu)} &= \left|\sum_{n\geq 1} \left(np_n - n\bar{r}^n p_0\right) \right| \\
&= \left|\sum_{n\geq 1}\sum_{0\leq k\leq n-1} n\bar{r}^{n-k-1}\left(p_{k+1} - \bar{r}p_k\right)\right| \\
&\leq \sum_{k\geq 0} |p_{k+1} - \bar{r}p_k| \sum_{n\geq k+1} n\bar{r}^{n-k-1} \\
&\leq \tilde{C}\sum_{k\geq 0} (k+1)\frac{|p_{k+1} - \bar{r}p_k|}{\sqrt{|p_{k+1} + \bar{r}p_k|}}\sqrt{|p_{k+1} + \bar{r}p_k|} \\
&\leq C\sqrt{\DD},
\end{align*}
where the last inequality follows from the observation that for any $a,b \in \mathbb R_+$,
\begin{equation}\label{eq:elementary}
(a-b)\left(\log a - \log b\right) \geq \frac{(a-b)^2}{a+b}.
\end{equation}
This finishes the proof. \qed

Equipped with Lemma \ref{lem1}, we can prove our first entropy-entropy dissipation estimate.
\begin{theorem}\label{thm1}
Assume that ${\bf p} \in V_\mu$ is the unique solution to the nonlinear system of ODEs \eqref{eq:law_limit}. There exists a fixed constant $C > 0$ such that \begin{equation}\label{eq:EED1}
\HH \leq C\sqrt{\DD}\left|\log(1/\DD)\right|.
\end{equation}
for all $t > 0$. Consequently, there exists some $C > 0$ such that
\begin{equation}\label{eq:polynoimal_decay}
\HH = \DD_{\KL}({\bf p}(t)~||~{\bf p}^*) \leq \frac{C}{t + C}.
\end{equation}
\end{theorem}

\Proof We start from the decomposition
\begin{equation}\label{eq:decompose_H}
\HH = \sum_{n \geq 0} \left(p_n\,\log \frac{p_n}{p^*_n} + p^*_n - p_n\right) := \HH_1 + \HH_2,
\end{equation}
with
\begin{equation}\label{eq:H1}
\HH_1 := \sum_{n \colon p_n \geq Mp^*_n} \left(p_n\,\log \frac{p_n}{p^*_n} + p^*_n - p_n\right)
\end{equation}
and
\begin{equation}\label{eq:H2}
\HH_2 := \sum_{n \colon p_n < Mp^*_n} \left(p_n\,\log \frac{p_n}{p^*_n} + p^*_n - p_n\right)
\end{equation}
for some (potentially) large $M > 0$ to be determined. We can take $K > 1$ so that \[\sup\limits_{t > 0} \sum_{n \geq 0} K^n p_n(t) < \infty,\] and then let $\theta > 0$ to be small enough for which $(r^*)^{2\theta} > 1/K$. The estimate of $\HH_1$ can be carried out as follows:
\begin{align*}
\HH_1 &= \sum_{n \colon p_n \geq Mp^*_n} \left(p_n\,\log \frac{p_n}{p^*_n} + p^*_n - p_n\right) \\
&\leq C_1\sum_{n \colon p_n \geq Mp^*_n} p_n(n + C_2) \\
&\leq \frac{C_1}{M^\theta}\sum_{n \colon p_n \geq Mp^*_n} K^{\frac{n}{2}}p_n \left(\frac{p_n}{p^*_n}\right)^\theta \frac{n+C_2}{K^{\frac{n}{2}}}\\
&\leq \frac{C_3}{M^\theta}\left(\sum_{n\geq 0} K^n p^{2+2\theta}_n \right)^{\frac 12}\left(\sum_{n\geq 0} \frac{(n+C_2)^2}{\left(K(r^*)^{2\theta}\right)^n}\right)^{\frac 12} \leq C/M^\theta,
\end{align*}
in which $C_1, C_2, C_3 > 0$ and $C >0$ are some fixed constants. Next, notice that for each $n \in \mathbb N$, we have
\begin{equation}\label{eq:identity}
p_n - p^*_n = \sum_{0\leq k\leq n-1} \bar{r}^{n-k-1}\left(p_{k+1} - \bar{r}p_k\right) + \left(\bar{r}^n p_0 - (r^*)^n p^*_0\right).
\end{equation}
Then for some constant $C > 0$ whose value might change from line to line, \begin{align*}
\sum_{n\geq 0} \frac{|p_n - p^*_n|^2}{p_n + p^*_n} &\leq \sum_{n\geq 0} |p_n - p^*_n| \\
&\leq \sum_{k \geq 0} |p_{k+1}-\bar{r}p_k|\sum_{n\geq k+1} \bar{r}^{n-k-1} + \sum_{n\geq 0} \left|\bar{r}^n p_0 - (r^*)^n p^*_0\right| \\
&\leq C|p_0 - p^*_0| + C\left(\sum_{k\geq 0} \frac{|p_{k+1}-\bar{r}p_k|^2}{p_{k+1}+\bar{r}p_k}\right)^{\frac 12} \leq C\sqrt{\DD},
\end{align*}
where the last inequality follows from Lemma \ref{lem1} and the elementary fact \eqref{eq:elementary}. To control $\HH_2$, we observe that for any $a,b \in \mathbb{R}_+$, if $a \leq Mb$ for some $M > 0$, then $(a-b)\log(a/b) + b - a \leq C\,\left|\log M\right|\,\frac{|a-b|^2}{a+b}$ for some constant $C > 0$ (independent of $a,b$ and $M$). Thanks to this observation, we deduce that
\begin{equation}\label{eq:bound_H2}
\HH_2 \leq C\,\left|\log M\right|\sum_{n\geq 0} \frac{|p_n - p^*_n|^2}{p_n + p^*_n} \leq C\,\sqrt{\DD}\left|\log M\right|.
\end{equation}
Putting the estimates on $\HH_1$ and $\HH_2$ together, we obtain
\begin{equation}\label{eq:bound_H}
\HH = \HH_1 + \HH_2 \leq \frac{C}{M^\theta} + C\sqrt{\DD}\left|\log M\right|
\end{equation}
for some constant $C > 0$, from which the advertised bound \eqref{eq:EED1} on $\HH$ follows by taking $M = (1/\DD)^{1/\theta}$. Finally, \eqref{eq:polynoimal_decay} follows from the simple differential inequality \[\frac{\dd}{\dd t} \HH = -\DD \leq -C\frac{\HH^2}{(\log \HH)^2}\] for some constant $C > 0$. \qed
\subsection{The refined entropy-entropy dissipation estimate}\label{subsec:2.3}
Thanks to the decay of the relative entropy \eqref{eq:polynoimal_decay}, there exists some $t_* > 0$ such that $|\bar{r}(t) - r^*|$ is sufficiently small for all $t\geq t_*$. In particular, we may assume that for $t \geq t_*$, $\frac{r^*}{\bar r}\leq \frac{1+K}{2}$ and $\frac{\bar r}{r^*}\leq \frac{1+K}{2}$ for some $K > 1$ thanks to Corollary \ref{cor1}. These observations enable us to refine the entropy-entropy dissipation estimate \eqref{eq:EED1} at large times, leading to an improved convergence guarantee than \eqref{eq:polynoimal_decay} when time $t$ becomes large. The main result of this section is summarized in the following theorem, which can be viewed as the complete version of Theorem \ref{thm} provided in the beginning of section \ref{sec:2}.

\begin{theorem}\label{thm2}
Assume that ${\bf p} \in V_\mu$ is the unique solution to the nonlinear system of ODEs \eqref{eq:law_limit}. Define the time $t^*$ such that for all $t\geq t^*$
\[
\frac{r^*}{\bar r}\leq \frac{1+K}{2} ~~\text{and}~~ \frac{\bar r}{r^*}\leq \frac{1+K}{2},
\]
with $K$ given by Corollary~\ref{cor1}. We then have for all $t\geq t^*$ that
\[
\HH \leq C\,\DD\,\left|\log(1/\DD)\right|
\]
for some positive constant~$C$ independent of $t$. Consequently, there exists some $C > 0$ such that
\begin{equation}\label{eq:almost_exponential_decay}
\HH = \DD_{\KL}({\bf p}(t)~||~{\bf p}^*) \leq C\expo^{-C\sqrt{t}}
\end{equation}
whenever $t \geq t_*$. In particular, taking into account of the estimate \eqref{eq:polynoimal_decay}, \eqref{eq:almost_exponential_decay} holds for all $t > 0$ (with a possibly different constant $C > 0$).
\end{theorem}

\begin{proof}
  We adapt the ideas of~\cite{canizo_trend_2017} here, and more specifically their Theorem~2.3, which shows an equivalence between the validity of certain weighted logarithmic Sobolev inequalities and a Muckenhoupt-type criterion.

  For any given $N$,  we introduce an interpolated equilibrium
  \[
  q_n=\left\{\begin{aligned}
  &p_n^*,\quad n\leq N,\\
  &\frac{\bar p_n}{n},\quad n>N,
  \end{aligned}\right.
  \]
  where we define $\bar p_n=p_0\,(\bar r)^n$. We accordingly introduce the interpolated relative entropy with respect to $\{\bar p_n\}_{n\geq 0}$
  \[
H_{int}=\sum_n \frac{q_n}{\bar q}\,\frac{p_n}{p_n^*}\,\log \frac{p_n}{m\,p_n^*},
\]
where we normalized the various measures
\[
\bar q=\sum_{n\geq 0} q_n,\quad m=\sum_{n\geq 0} \frac{q_n}{\bar q}\,\frac{p_n}{p_n^*}.
\]
We immediately remark that by Proposition~\ref{prop1}, there exists some constant~$L>1$ uniform in time such that
\begin{equation}
|\bar q-1|\leq \sum_{n>N} (p_n^*+\bar p_n/n)\leq C\,((r^*)^N+\bar r^N)\leq \frac{C}{L^N}.\label{q-1}
\end{equation}
Furthermore
\begin{equation}
\begin{split}
  |m-1|&\leq \frac{1}{\bar q}\,\sum_{n\geq 0} \frac{|q_n-p_n^*|}{p_n^*}\,p_n+|1-1/\bar{q}|\\
  &\leq \frac{C}{L^N}+\frac{1}{\bar q}\,\sum_{ n\geq N} (1+\bar p_n/p_n^*)\,p_n\\
  & \leq \frac{C}{L^N}+C\,\sum_{n\geq N} \left(1+\left(\frac{1+K}{2}\right)^n\right)\,p_n\leq \frac{C}{L^N},
\end{split}\label{m-1}
\end{equation}
for some constant $L>1$, thanks to Proposition~\ref{prop1} and Corollary~\ref{cor1}.

We now observe that, since $\bar r\,\bar p_n=\bar p_{n+1}$,
\[
\begin{split}
  \DD&=\sum_{n\geq 0} (\bar r\,p_n-p_{n+1})\,(\log (\bar r\,p_n)-\log p_{n+1})\\
  &=\sum_{n\geq 0} \bar p_{n+1}\,(p_n/\bar p_n-p_{n+1}/\bar p_{n+1})\,(\log (p_n/\bar p_n) )-\log (p_{n+1}/\bar p_{n+1}))\\
  &\geq \sum_{n\geq 0} \bar p_{n+1}\,\left|\sqrt{\frac{p_n}{\bar p_{n}}}-\sqrt{\frac{p_{n+1}}{\bar p_{n+1}}} \right|^2,
\end{split}
\]
due to the elementary inequality $(a-b)\,(\log a-\log b)\geq (\sqrt{a}-\sqrt{b})^2$ for all $a,b > 0$.

We denote accordingly
\[
\bar{\DD} =\sum_{n\geq 0} \bar p_{n+1}\,\left|\sqrt{\frac{p_n}{\bar p_{n}}}-\sqrt{\frac{p_{n+1}}{\bar p_{n+1}}} \right|^2 ~~\text{and}~~ {\DD}^*=\sum_{n\geq 0} \bar p_{n+1}\,\left|\sqrt{\frac{p_n}{p_{n}^*}}-\sqrt{\frac{p_{n+1}}{p_{n+1}^*}} \right|^2,
\]
so we immediately have that $\bar{\DD} \leq \DD$. Moreover, denote $f_n=\sqrt{p_n/p_n^*}$, $g_n=\sqrt{p_n/\bar p_n}$ and observe that
\[
\begin{split}
  {\DD}^*-\bar{\DD}&=\sum_{n\geq 0} \bar p_{n+1}\,\left(f_n-f_{n+1}+g_n-g_{n+1}\right)\,(f_n-g_n-(f_{n+1}-g_{n+1}))\\
  &\leq C\,\left(\sqrt{\bar{\DD}}+\sqrt{{\DD}^*}\right)\,\left(\sum_{n\geq 0} \bar p_{n+1}\,|f_n-g_n|^2\right)^{1/2},
\end{split}
\]
by Cauchy-Schwarz inequality and using that $(a+b)^2\leq 2(a^2+ b^2)$ together with $\bar p_{n+2}=\bar r\,\bar p_{n+1}$.

On the other hand, by the definitions of $\bar p_n$ and $p_n^*$,
\[
|f_n-g_n|=\frac{\sqrt{p_n}\,|\bar p_n-p_n^*|}{\sqrt{\bar p_n}\,\sqrt{p_n^*}\,(\sqrt{\bar p_n}+\sqrt{p_n^*})}\leq C\,(1+n)\,\frac{\sqrt{p_n}\,(\bar p_n+p_n^*)}{\sqrt{\bar p_n}\,\sqrt{p_n^*}\,(\sqrt{\bar p_n}+\sqrt{p_n^*})} |\bar r-r^*|
\]
By Lemma~\ref{lem1},
\[
\sum_{n\geq 0} \bar p_{n+1}\,|f_n-g_n|^2\leq C\,\DD\,\sum_{n\geq 0} (1+n)^2\,p_n\,\left(1+\frac{\bar p_n}{p_n^*}\right).
\]
Using the assumptions of Theorem~\ref{thm2} and Corollary~\ref{cor1}, we finally obtain
\[
\sum_{n\geq 0} \bar p_{n+1}\,|f_n-g_n|^2\leq C\,\DD\,\sum_{n\geq 0} (1+n)^2\,p_n\,\left(1+\left(\frac{1+K}{2}\right)^n\right)\leq C\,\DD.
\]
This implies that, since $\bar{\DD} \leq \DD$,
\[
{\DD}^* = \bar{\DD} + {\DD}^*-\bar{\DD} \leq \DD+C\,\sqrt{\DD}\,(\sqrt{\DD}+\sqrt{{\DD}^*}),
\]
or
\begin{equation}
{\DD}^*\leq C\,{\DD}.\label{D*D}
  \end{equation}
We are now ready to apply Theorem 2.3 of~\cite{canizo_trend_2017} with $f_n=\sqrt{p_n/p_{n}^*}$, $\mu_n=q_n/\bar q$, $\nu_n=\bar p_{n+1}$. We estimate
\[
\sum_{n=0}^k \frac{1}{\nu_n}=\frac{1}{p_0}\,\sum_{n=0}^k \frac{1}{\bar r^{n+1}}\leq \frac{C}{\bar r^{k+1}},
\]
while
\[
\sum_{n=k+1}^\infty \mu_n \leq C\,\left(\sum_{n=k+1}^N p_n^*+\sum_{n=N+1}^\infty \frac{\bar p_n}{n}\right) \leq C\,((r^*)^{k+1}\,\mathbbm{1}_{k<N}+\bar r^{k+1}/(k+1)).
\]
Denoting $\Psi(x)=|x|\,\log(1+|x|)$, we hence have that
\[\Psi^{-1}\left(\frac{1}{\sum_{n=k+1}^\infty \mu_n}\right)\geq \frac{1}{C}\,\min\left(\frac{1}{(r^*)^{k+1}}\,\mathbbm{1}_{k<N}, \frac{k+1}{\bar r^{k+1}}\right).\]

\noindent Therefore the constant \[B_1 := \sup\limits_{k\geq m} \frac{\sum_{n=0}^k \frac{1}{\nu_n}}{\Psi^{-1}\left(\frac{1}{\sum_{n=k+1}^\infty \mu_n}\right)}\] introduced in~\cite{canizo_trend_2017} is bounded by
\[B_1\leq C\,\sup_k\, \max\left(\frac{1}{k+1},\, \frac{(r^*)^{k+1}}{\bar r^{k+1}}\,\mathbbm{1}_{k<N}\right)\leq C\,(1+N\,(r^*)^N/\bar r^N).\]
From the definition of $\mu_n$ and~\eqref{q-1}, we know that $m$ is of order $1$ so that the constant \[B_2 := \frac{\sum_{n=0}^{m-1} \frac{1}{\nu_n}}{\Psi^{-1}\left(\frac{1}{\sum_{n=0}^{m-1} \mu_n}\right)} \] defined in Theorem 2.3 of \cite{canizo_trend_2017} has the bound $B_2\leq C$. Hence by Theorem 2.3 in~\cite{canizo_trend_2017}
\begin{equation}
H_{int} \leq C\,(1+N\,(r^*)^N/\bar r^N)\,{\DD}^*\leq C\,(1+N\,(r^*)^N/\bar r^N)\,\DD,\label{HiD}
\end{equation}
where the last inequality follows from \eqref{D*D}.

We may also bound $\HH$ in terms of $H_{int}$. Indeed, from the definition of $q_n$
\[
\begin{split}
  \HH&=\sum_{n\geq 0} p_n^*\,\left(\frac{p_n}{p_n^*}\,\log \frac{p_n}{p^*_n}+1-\frac{p_n}{p_n^*}\right)\leq \sum_{n\leq N} q_n\,\left(\frac{p_n}{p_n^*}\,\log \frac{p_n}{p^*_n}+1-\frac{p_n}{p_n^*}\right)+C\,\sum_{n>N} n\,p_n\\
  &\leq  \sum_{n\geq 0} q_n\,\left(\frac{p_n}{p_n^*}\,\log \frac{p_n}{p^*_n}+1-\frac{p_n}{p_n^*}\right)+\frac{C}{L^N},
  \end{split}
\]
for some constant $L>1$, using again Corollary~\ref{cor1} and the fact that $x\,\log x+1-x\geq 0$ for all $x\geq 0$.

Furthermore, according to the definitions of $m$ and $H_{int}$
\[
\begin{split}
&  \sum_{n\geq 0} \frac{q_n}{\bar q}\,\left(\frac{p_n}{p_n^*}\,\log \frac{p_n}{p^*_n}+1-\frac{p_n}{p_n^*}\right)=\sum_{n\geq 0} \frac{q_n}{\bar q}\,\frac{p_n}{p_n^*}\,\log \frac{p_n}{p^*_n}+1-m\\
  &\qquad=\sum_{n\geq 0} \frac{q_n}{\bar q}\,\frac{p_n}{p_n^*}\,\log \frac{p_n}{m\,p^*_n}+m\,\log m+1-m=H_{int}+m\,\log m+1-m.
\end{split}
\]
Note as well that by~\eqref{m-1}, $m\,\log m+1-m\leq C\,|m-1|^2\leq \frac{C}{L^N}$, so that eventually from~\eqref{HiD}
\begin{equation}
\HH\leq \frac{H_{int}}{\bar q}+\frac{C}{L^N}\leq C\,(1+N\,(r^*)^N/\bar r^N)\,\DD+\frac{C}{L^N}.\label{HD}
\end{equation}
It mostly remains to optimize in $N$ in~\eqref{HD}. For this we note that from Lemma~\ref{lem1}
\[
(r^*)^N/\bar r^N=\exp\left(N\,\log\left(1+\frac{r^*-\bar r}{\bar r}\right)\right)\leq \exp\left(C\,N\,\frac{\sqrt{\DD}}{\bar r}\right).
\]
Therefore as long as we restrict ourselves to $N\leq \frac{C}{\sqrt{\DD}}$, we have that $(r^*)^N/\bar r^N\leq C$. We hence choose $N=N_0\,\left|\log (1/\DD)\right|$, which satisfies this constraint together with $C/L^N\leq \DD$ provided $N_0$ is large enough. We then find
\[
\HH\leq C\,\DD\,\left|\log(1/\DD)\right|,
\]
as claimed. Finally, \eqref{eq:almost_exponential_decay} follows from the basic differential inequality \[\frac{\dd}{\dd t} \HH = -\DD \leq -C\frac{\HH}{\left|\log (1/\HH)\right|},\] which completes the proof.
\end{proof}

\begin{remark}
In literature, an almost-exponential decay result to equilibrium of the type \eqref{eq:almost_exponential_decay} is sometimes called Jabin
and Niethammer’s rate of decay to equilibrium \cite{canizo_trend_2017}, owing to their work \cite{jabin_rate_2003} on quantitative entropy method for the Becker–D\"oring equations.
\end{remark}

We end this section with several numerical experiments demonstrating the entropic convergence of ${\bf p}(t)$ to ${\bf p}^*$. We use $\mu = 10$ for the model. For the discretization of the nonlinear ODE system \eqref{eq:law_limit}, we use the first 501 components (i.e., $(p_0(t),\ldots,p_{500}(t)$) to characterize the distribution ${\bf p}(t)$. We will test our bound \eqref{eq:almost_exponential_decay} using two 'typical' initial conditions to be described in detail below. The classical fourth-order Runge-Kutta method is used to discretize the ODE system \eqref{eq:law_limit} with the time step $\Delta t = 0.01$.

We plot in figure \ref{fig:entropic_decay}-left the evolution of the relative entropy $\DD_{\KL}({\bf p(t)}~||~{\bf p}^*)$ over $0\leq t \leq 200$, starting from the initial datum $p_\mu(0) = 1$ and $p_i(0) = 0$ for $i\neq \mu$. This choice of initial wealth distribution corresponds to the initial configuration prescribed by $S_i(0) = \mu$ for all $1\leq i\leq N$ at the level of $N$-agents system. We observe that the ansatz
\begin{equation}\label{eq:ansatz}
\DD_{\KL}({\bf p(t)}~||~{\bf p}^*) \approx C_1\,\expo^{-C_2\sqrt{t}},
\end{equation}
where $C_1$ and $C_2$ are positive constants, fits extremely well with the evolution of the relative entropy over the time span under investigation. Figure \ref{fig:entropic_decay}-right records the evolution of the relative entropy over $0\leq t \leq 2000$, now starting from the initial distribution $p_0(0) = 1 - \frac{\mu}{500} = 0.98$, $p_{500}(0) = 0.02$, and $p_i(0) = 0$ for $i \notin \{0,500\}$. This choice of initial distribution can be linked to the initial configuration given by $S_i(0) = 0$ for all $1\leq i\leq N-1$ and $S_N(0) = N\cdot \mu = 5000$ in the $N$-agents system with $N = 500$, in which case we put all the dollars in the hand of a single agent initially. Again, the ansatz \eqref{eq:ansatz} fits extremely well with the evolution of the relative entropy in our numerical simulation.

\begin{figure}[!htb]
  \begin{subfigure}{0.47\textwidth}
    \centering
    \includegraphics[scale=0.6]{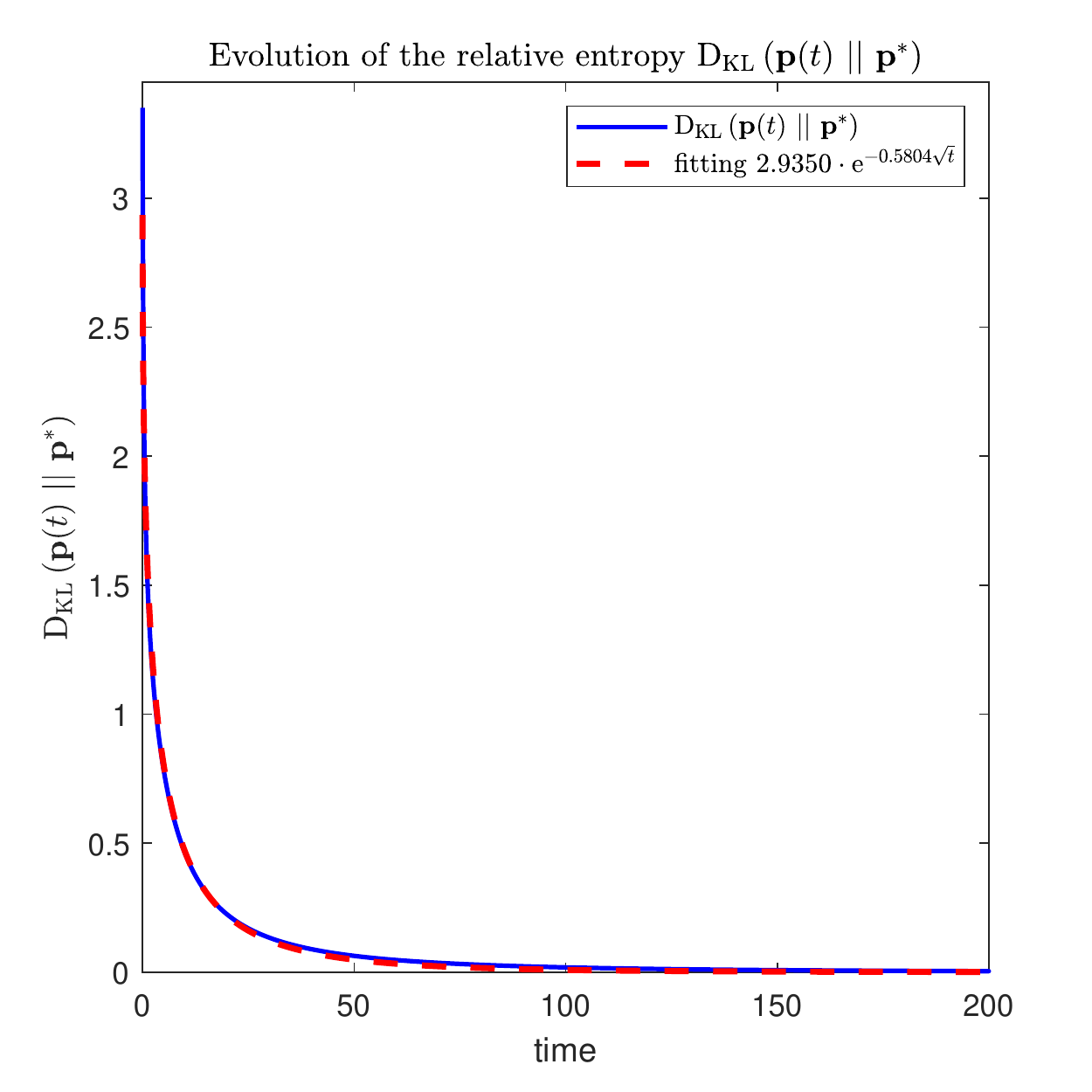}
  \end{subfigure}
  \hspace{0.1in}
  \begin{subfigure}{0.47\textwidth}
    \centering
    \includegraphics[scale=0.6]{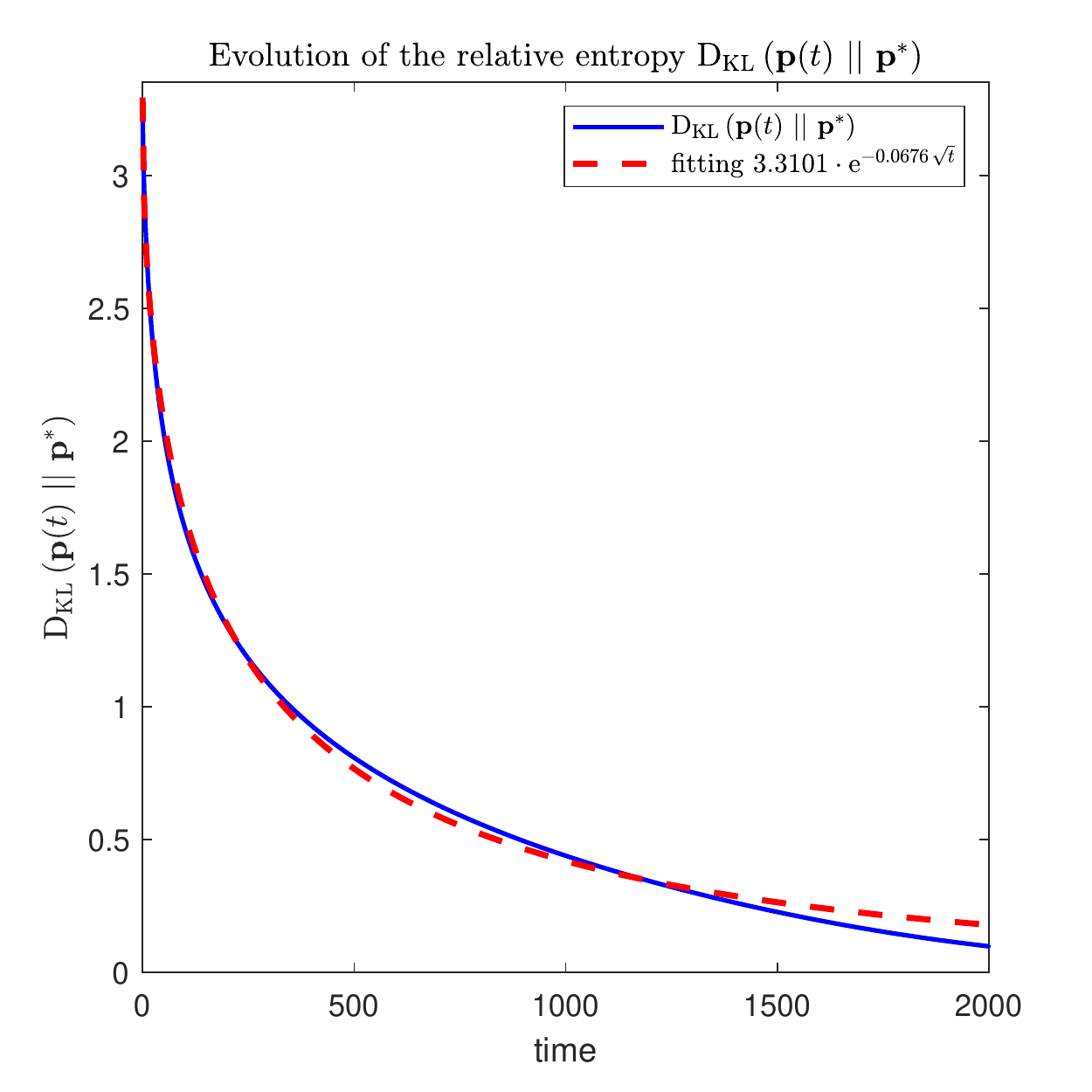}
  \end{subfigure}
  \caption{{\bf Left}: Evolution of the relative entropy $\DD_{\KL}({\bf p(t)}~||~{\bf p}^*)$ over $0\leq t \leq 200$, starting from the initial datum $p_\mu(0) = 1$ and $p_i(0) = 0$ for $i\neq \mu$. {\bf Right}: Evolution of the relative entropy over $0\leq t \leq 2000$, starting from the initial condition $p_0(0) = 1 - \frac{\mu}{500} = 0.98$, $p_{500}(0) = 0.02$, and $p_i(0) = 0$ for $i \notin \{0,500\}$. The decay is of the form \eqref{eq:ansatz} in both cases as predicted by Theorem \ref{thm2}.}
  \label{fig:entropic_decay}
\end{figure}

\section{Uniform-in-time propagation of chaos}
\setcounter{equation}{0}

Let us introduce the empirical measure $\rho_{\emp}$ of the unbiased exchange dynamics \eqref{unbiased_exchange}, denoted by
\begin{equation}\label{eq:empirical_measure}
\rho_{\emp}(t) = \frac{1}{N}\sum_{1\leq i\leq N} \delta_{S_i(t)}.
\end{equation}
Recall that ${\bf q} = (q_0,q_1,\ldots)$ represents the individual coordinates of the measure $\rho_{\emp}$ and $r := 1 - q_0$. For each test function $\psi$, we have
\begin{equation}\label{eq:weak_formulation}
\begin{aligned}
&\frac{\dd}{\dd t} \mathbb{E}\left[\psi({\bf q})\right] = \\
&\mathbb{E}\left[\sum_{\ell \geq 1, m\geq 0} Nq_\ell\left(q_m - \frac{1}{N}\mathbbm{1}\{\ell = m\}\right)\left(\psi\left({\bf q} + \frac{1}{N}\left(\delta_{\ell-1} + \delta_{m+1} - \delta_\ell - \delta_m\right)\right) - \psi({\bf q})\right)\right].
\end{aligned}
\end{equation}
The main result (see Theorem \ref{thm4}) we want to prove in this section is that for all $t > 0$ and all $N \geq 1$, we have that
\begin{equation}\label{eq:L1chaos_restate}
\mathbb{E}\,\|{\bf q}(t) - {\bf p}(t)\|^2_{\ell^1} \leq \frac{C}{1+\log N}.
\end{equation}
By an interpolation argument, such $\ell^1$ type uniform in time propagation of chaos also leads to a entropic uniform in time propagation of chaos, i.e.,
\begin{equation}\label{eq:entropic_chaos_restate}
\mathbb{E}\,\sum_n q_n(t)\,\log\frac{q_n(t)}{p_n(t)}\leq C\,\frac{\log \log N}{1+\log N}.
\end{equation}

\subsection{Mimicking the continuous entropy dissipation}\label{subsec:3.1}
The first step is to derive an appropriate entropy dissipation result by showing:
\begin{lemma}
For any solution to the unbiased exchange dynamics~\eqref{unbiased_exchange}, we have that
\[
\frac{\dd}{\dd t} \mathbb{E}\left[\sum_{n\geq 0} q_n\log q_n\right]\leq \sum_{\ell\geq 1,\,m\geq 0}q_\ell\,q_m\,\log \frac{(q_{\ell-1}+1/N)\,(q_{m+1}+1/N)}{q_\ell\,q_m}+\frac{C\,\log N}{\sqrt{N}},
\]
and that
\[
\begin{split}
  &\frac{\dd}{\dd t} \mathbb{E}\left[\sum_{n\geq 0} q_n\log q_n\right]\\
&\quad\leq -\mathbb{E}\,\sum_{n\geq 0} (\bar r\,(q_n+1/N)-(q_{n+1}+1/N))\,L\left(\frac{\bar r\,(q_n+1/N)}{q_{n+1}+1/N}\right)
  +\frac{C\,\log N}{\sqrt{N}},
\end{split}
\]
where we recall that $\bar r=\sum_{n\geq 1} q_n$ and we define the function $L(x)$ as a truncated version of the natural logarithmic function:
\[
L(x)=\log x\quad\mbox{if}\ \frac{1}{\expo^2}\leq x\leq \expo^2,\quad L(x)=2\quad\mbox{if}\  x\geq \expo^2,\quad L(x)=-2\quad\mbox{if}\  x\leq \frac{1}{\expo^2}.
\]
\label{agentsentropy}
\end{lemma}
\begin{proof}
  Inserting $\psi({\bf q}) := \sum_{n\geq 0} h(q_n)$ with $h(x) = x\log x$ for $x > 0$ and $h(0) = 0$ in \eqref{eq:weak_formulation}, we obtain
\begin{equation}\label{eq:stochastic_entropy}
\begin{aligned}
\frac{\dd}{\dd t} \mathbb{E}\left[\sum_{n\geq 0} q_n\log q_n\right] &= N\sum_{\ell \geq 1, m\geq 0} \mathbb{E}\,q_\ell\, q_m\,\left(\psi\left({\bf q} + \frac{1}{N}\left(\delta_{\ell-1} + \delta_{m+1} - \delta_\ell - \delta_m\right)\right) - \psi({\bf q})\right)\\ 
&\quad+ R_N^1,
\end{aligned}
\end{equation}
where
\[
\begin{aligned}
  R_N^1&=-\mathbb{E}\,\sum_{\ell\geq 1} q_\ell\, \left( \psi\left( {\bf q}+\frac{1}{N}(\delta_{\ell-1}+\delta_{\ell+1}-2\delta_\ell)\right)-\psi({\bf q}) \right)\\
  &=-\mathbb{E}\,\sum_{\ell\geq 1} q_\ell\, \left(h(q_{\ell-1}+1/N)-h(q_{\ell-1})+h(q_{\ell+1}+1/N)-h(q_{\ell+1})+h(q_{\ell}-2/N)-h(q_{\ell})\right).
\end{aligned}
\]
Observe that for $\eps=k/N$ for $k=\pm 1$ or $k=-2$, we have that $|\eps|\leq C\,q_n$ if  $q_n\neq 0$ (as $q_n$ are necessarily in increments of $1/N$). Hence if $|\eps|\leq q_n/2$ then
\[
|h(q_n+\eps)-h(q_n)|=\left|\eps\,\log(q_n+\eps)+q_n\,\log\frac{q_n+\eps}{q_n}\right|\leq |\eps|\,\log N+C\,q_n\,\frac{|\eps|}{q_n}\leq C\,\frac{\log N}{N},
\]
as $|\eps|/q_n\leq 1/2$ and $\log (q_n+\eps)\geq \log (1/N) = -\log N$. On the other hand for $q_n\leq 2\,|\eps|$, we simply have that
\[
|h(q_n+\eps)-h(q_n)|\leq |h(q_n+\eps)|+|h(q_n)|\leq C\,\frac{\log N}{N}.
\]
In both cases, this allows us to write that
\begin{equation}
|R_N^1|\leq C\,\frac{\log N}{N}\,\sum_{\ell\geq 1} q_\ell \leq C\,\frac{\log N}{N}.\label{boundR1N}
\end{equation}
We now further simplify \eqref{eq:stochastic_entropy} by using Taylor expansion. For fixed $l,\,m$, we denote $\eps_{\ell,m} =\frac{1}{N}\left(\delta_{\ell-1} + \delta_{m+1} - \delta_\ell - \delta_m\right)$ and define the probability mass $\widehat{{\bf q}}^{\,\ell,m} = {\bf q} + \eps_{\ell,m}$, then we observe that
\[
\begin{aligned}
  &\psi\left( {\bf q}+\frac{1}{N}(\delta_{\ell-1}+\delta_{m+1}-\delta_\ell-\delta_m)\right)-\psi({\bf q})=\sum_n (h(\widehat{q}^{\,\ell,m}_n)-h(q_n)).\\
\end{aligned}
\]
Our goal is to bound
\begin{equation}
\sum_n (h(\widehat{q}^{\,\ell,m}_n)-h(q_n))\leq \frac{1}{N}\,\log \frac{(q_{\ell-1}+1/N)\,(q_{m+1}+1/N)}{q_\ell\,q_m}+r_N,\label{goal1}
\end{equation}
for some vanishing $r_N$.

We need to study separately the generic case where $\ell\neq m$, $\ell\neq m+1$ and the two special cases $\ell=m$ and $\ell=m+1$ (or equivalently $m=\ell-1$).

\smallskip

$\bullet$ {\em Case $\ell=m+1$.} This is the trivial case as we also have that $m=\ell-1$. That implies that $\eps_{\ell,m}=0$ so that
\[
\sum_n (h(\widehat{q}^{\,\ell,m}_n)-h(q_n)) =0=\frac{1}{N}\,\log \frac{q_{m}\,q_{\ell}}{q_\ell\,q_m}= \frac{1}{N}\,\log \frac{q_{\ell-1}\,q_{m+1}}{q_\ell\,q_m},
\]
whence we can simply take $r_N=0$.

\smallskip

$\bullet$ {\em Case $\ell\neq m$, and $\ell\neq m+1$.} We have
\[
\begin{split}
  &\sum_n (h(\widehat{q}^{\,\ell,m}_n)-h(q_n)) =(q_{\ell-1}+1/N)\,\log (q_{\ell-1}+1/N)-q_{\ell-1}\,\log q_{\ell-1}+\\
  &+(q_{\ell}-1/N)\,\log (q_{\ell}-1/N)-q_\ell\,\log q_\ell+(q_{m+1}+1/N)\,\log (q_{m+1}+1/N)-q_{m+1}\,\log q_{m+1}\\
  &+(q_{m}-1/N)\,\log (q_{m}-1/N)-q_m\,\log q_m.
\end{split}\]
Therefore, we can handle separately each difference and we only have to show for example that
\[
\begin{split}
&  (q_{\ell-1}+1/N)\,\log (q_{\ell-1}+1/N)-q_{\ell-1}\,\log q_{\ell-1}+(q_{\ell}-1/N)\,\log (q_{\ell}-1/N)-q_\ell\,\log q_\ell\\
  &\qquad\leq \frac{1}{N}\,\log \frac{q_{\ell-1}+1/N}{q_\ell}+r_N.
\end{split}
\]
Note that $1+\log x$ is increasing so that for $a\geq b$
\[
a\,\log a-b\,\log b=\int_{b}^a (1+\log x)\,dx\leq (1+\log a)\,(a-b).
\]
Hence
\[
(q_{\ell-1}+1/N)\,\log (q_{\ell-1}+1/N)-q_{\ell-1}\,\log q_{\ell-1} \leq \frac{1}{N}\,(1+\log (q_{\ell-1}+1/N)),
\]
and similarly
\[
(q_{\ell}-1/N)\,\log (q_{\ell}-1/N)-q_\ell\,\log q_\ell\leq -\frac{1}{N}\,(1+\log (q_\ell-1/N)).
\]
As long as $q_\ell\geq 2/N$, we can further write
\[
(q_{\ell}-1/N)\,\log (q_{\ell}-1/N)-q_\ell\,\log q_\ell\leq -\frac{1}{N}\,(1+\log q_\ell)+\frac{C}{N^2\,q_\ell}.
\]
We do not need to consider the case where $q_\ell=0$ as we have a $q_\ell\,q_m$ factor in front of everything. Further, if $q_\ell=1/N$ then
\[
(q_{\ell}-1/N)\,\log (q_{\ell}-1/N)-q_\ell\,\log q_\ell= q_\ell\,\log N.
\]
Consequently, we may write
\[
\begin{split}
&  (q_{\ell-1}+1/N)\,\log (q_{\ell-1}+1/N)-q_{\ell-1}\,\log q_{\ell-1}+(q_{\ell}-1/N)\,\log (q_{\ell}-1/N)-q_\ell\,\log q_\ell\\
  &\qquad\leq \frac{1}{N}\,\log \frac{q_{\ell-1}+1/N}{q_\ell} + C\,q_\ell\,\log N\,\mathbbm{1}\{q_\ell = 1/N\} + \frac{C}{N^2\,q_\ell},
\end{split}
\]
leading to \[r_N=C/(N^2\,q_\ell) + C/(N^2\,q_m) + C\,q_\ell\,\log N\,\mathbbm{1}\{q_\ell = 1/N\} + C\,q_m\,\log N\,\mathbbm{1}\{q_m = 1/N\}.\]

\smallskip

$\bullet$ {\em Case $\ell=m$.} This is the second special case and we have
\[
\begin{split}
  &\sum_n (h(\widehat{q}^{\,\ell,m}_n)-h(q_n)) =(q_{\ell-1}+1/N)\,\log (q_{\ell-1}+1/N)-q_{\ell-1}\,\log q_{\ell-1}+\\
  &+(q_{\ell}-2/N)\,\log (q_{\ell}-2/N)-q_\ell\,\log q_\ell+(q_{\ell+1}+1/N)\,\log (q_{\ell+1}+1/N)-q_{\ell+1}\,\log q_{\ell+1}.\\
\end{split}
\]
This is handled in a pretty similar manner as in the previous case since we may write
\[
\begin{split}
  &(q_{\ell}-2/N)\,\log (q_{\ell}-2/N)-q_\ell\,\log q_\ell\leq -\frac{2}{N}\,(1+\log q_\ell)+\frac{C}{N^2\,q_\ell}\\
  &=-\frac{1}{N}\,(1+\log q_\ell)-\frac{1}{N}\,(1+\log q_m)+\frac{C}{N^2\,q_\ell}.
\end{split}
\]

\smallskip

Combing the three separate cases, we can make \eqref{goal1} specific to find
\begin{equation}
\begin{aligned}
\sum_n (h(\widehat{q}^{\,\ell,m}_n)-h(q_n)) &\leq \frac{1}{N}\,\log \frac{(q_{\ell-1}+1/N)\,(q_{m+1}+1/N)}{q_\ell\,q_m}+\frac{C}{N^2\,q_\ell}+\frac{C}{N^2\,q_m} \\
&\quad + C\,q_\ell\,\log N\,\mathbbm{1}\{q_\ell = 1/N\} + C\,q_m\,\log N\,\mathbbm{1}\{q_m = 1/N\}.\label{inter1}
\end{aligned}
\end{equation}
We further observe that
\[
\sum_{\ell,\;q_\ell>0} \frac{1}{N}\leq \frac{1}{\sqrt{N}}+\sum_{\ell\geq \sqrt{N},\;q_\ell>0}\frac{1}{N}\leq \frac{1}{\sqrt{N}}+\sum_{\ell\geq \sqrt{N}}q_\ell,
\]
since $q_\ell\geq 1/N$ if $q_\ell > 0$. Given that $\sum_\ell \ell\, q_\ell$ is constant this implies that
\begin{equation}
\sum_{\ell,\;q_\ell>0} \frac{1}{N}\leq \frac{1}{\sqrt{N}}+\frac{1}{\sqrt{N}}\,\sum_{\ell}\ell\,q_\ell\leq \frac{C}{\sqrt{N}}.\label{inter2}
\end{equation}
Combining \eqref{inter2} with \eqref{inter1} and the previous estimates \eqref{eq:stochastic_entropy}-\eqref{boundR1N}, we obtain that
\[
\frac{\dd}{\dd t} \mathbb{E}\left[\sum_{n\geq 0} q_n\log q_n\right]\leq \sum_{\ell\geq 1,\,m\geq 0} q_\ell\,q_m\,\log \frac{(q_{\ell-1}+1/N)\,(q_{m+1}+1/N)}{q_\ell\,q_m}+\frac{C\,\log N}{\sqrt{N}}.
\]
Of course we can trivially exclude those $\ell$ and $m$ s.t. $q_\ell=0$ or $q_m=0$. We may then decompose
  \[\begin{split}
  &  \sum_{\ell\geq 1,m\geq 0,\,q_\ell>0,\,q_m>0} q_\ell\,q_m\,\log\frac{(q_{\ell-1}+1/N)\,(q_{m+1}+1/N)}{q_\ell\,q_m}\\
  &\ =\sum_{\ell\geq 1,m\geq 0,\,q_\ell>0,\,q_m>0} q_\ell\,q_m\,\log\frac{q_{\ell-1}+1/N}{q_\ell}+\sum_{\ell\geq 1,m\geq 0,\,q_\ell>0,\,q_m>0} q_\ell\,q_m\,\log\frac{q_{m+1}+1/N}{q_m}\\
  &\ =\sum_{\ell\geq 1,\,q_\ell>0} q_\ell\,\log\frac{q_{\ell-1}+1/N}{q_\ell}+\bar r\,\sum_{m\geq 0,\,q_m>0} q_m\,\log\frac{q_{m+1}+1/N}{q_m},
\end{split}
  \]
  where we recall that $\bar r=\sum_{n\geq 1} q_n$.

We note that if $q_{\ell-1}=0$ and $q_{\ell}\geq C/N$ for example then $\log \frac{q_{\ell-1}+1/N}{q_\ell} \leq -\log C$. If $q_{\ell-1}=0$ and $q_{\ell}\leq C/N$ then we have that $\log \frac{q_{\ell-1}+1/N}{q_\ell}\geq -\log C$, and the corresponding terms are controlled by
\[
\log C\,\sum_{0<q_{\ell}\leq C/N} q_{\ell}\leq \frac{C}{\sqrt{N}}.
\]
Hence we can separate from the sums those $\ell$ s.t. $q_{\ell-1}=0$, and those $m$ s.t. $q_{m+1}=0$ to get by taking $C$ large enough
\[
\begin{split}
&  \frac{\dd}{\dd t} \mathbb{E}\left[\sum_{n\geq 0} q_n\log q_n\right]\leq \mathbb{E}\,\sum_{\ell\geq 1,q_\ell>0,q_{\ell-1}>0} q_\ell\,\log\frac{q_{\ell-1}+1/N}{q_\ell} +\frac{C\,\log N}{\sqrt{N}}\\
  &\quad+\mathbb{E}\,\bar r\,\sum_{m\geq 0,q_m>0,q_{m+1}>0} q_m\,\log\frac{q_{m+1}+1/N}{q_m}-3\,\mathbb{E}\,\sum_{\ell\geq 1,\;q_{\ell}>0,\;q_{\ell-1}=0} q_\ell\\
  &\quad-3\,\mathbb{E}\,\sum_{m\geq 0,\;q_{m}>0,\;q_{m+1}=0} q_m.
\end{split}
\]
Thus, through similar estimates as before, we can also deduce that
\[
\begin{split}
  &\frac{\dd}{\dd t} \mathbb{E}\left[\sum_{n\geq 0} q_n\log q_n\right]\leq \mathbb{E}\,\sum_{\ell\geq 1,q_\ell>0,q_{\ell-1}>0} (q_\ell+1/N)\,\log\frac{q_{\ell-1}+1/N}{q_\ell+1/N}\\
  &\quad+\mathbb{E}\,\bar r\,\sum_{m\geq 0,q_m>0,q_{m+1}>0} (q_m+1/N)\,\log\frac{q_{m+1}+1/N}{q_m+1/N}\\
  &-3\,\mathbb{E}\,\sum_{\ell\geq 1,\;q_{\ell}>0,\;q_{\ell-1}=0}  \left(q_{\ell}+\frac{1}{N}\right)-3\,\mathbb{E}\,\sum_{m\geq 0,\;q_{m}>0,\;q_{m+1}=0} \left(q_m+\frac{1}{N}\right)+\frac{C\,\log N}{\sqrt{N}}.
\end{split}
\]
 We can change the variables in the first and the third sum to $m=\ell-1$ and find
  \[\begin{split}
  & \frac{\dd}{\dd t} \mathbb{E}\left[\sum_{n\geq 0} q_n\log q_n\right]\\
  &\quad\leq -\mathbb{E}\,\sum_{m\geq 0,q_m>0,q_{m+1}>0} (\bar r\,(q_m+1/N)-(q_{m+1}+1/N))\, \log\frac{q_{m}+1/N}{q_{m+1}+1/N} \\
  &\quad -3\,\mathbb{E}\,\sum_{m\geq 0,\;q_{m}=0,\;q_{m+1}>0}  \left(q_{m+1}+\frac{1}{N}\right)-3\,\mathbb{E}\,\sum_{m\geq 0,\;q_{m}>0,\;q_{m+1}=0} \left(q_m+\frac{1}{N}\right)+\frac{C\,\log N}{\sqrt{N}}.
  \end{split}
  \]
  We further note that
  \[
  \begin{split}
    &\mathbb{E}\,\log \bar r\,\sum_{m\geq 0,q_m>0,q_{m+1}>0} (\bar r\,(q_m+1/N)-(q_{m+1}+1/N))\\
    &\quad=\mathbb{E}\,\bar r\,\log \bar r\,\sum_{m\geq 0,q_m>0,q_{m+1}>0} (q_m+1/N)-\mathbb{E}\,\log \bar r\, \sum_{m\geq 0,q_m>0,q_{m+1}>0} (q_{m+1}+1/N)\\
    &\quad= \mathbb{E}\,(\bar r\,\log \bar r)\,\left(1+\mathcal{O}(N^{-1/2})-\sum_{m\geq 0,q_m>0,q_{m+1}=0} (q_m+1/N)\right)\\
    &\qquad-\mathbb{E}\,\log \bar r\,\left(\bar r+\mathcal{O}(N^{-1/2})-\sum_{m\geq 0,q_m=0,q_{m+1}>0} (q_{m+1}+1/N)\right).
  \end{split}
  \]
Observe that $\log\bar r<0$ and $\bar r\,\log \bar r\geq -1$, further since $\sum_{n\geq 0} n\,q_n=\mu>0$ and $\sum_{n\geq 0} q_n=1$, we necessarily have that $q_0\leq 1-1/N$ or $\bar r\geq 1/N$ and $|\log \bar r|\leq \log N$. This yields that
  \[\begin{split}
  &\mathbb{E}\,\log \bar r\,\sum_{m\geq 0,q_m>0,q_{m+1}>0} (\bar r\,(q_m+1/N)-(q_{m+1}+1/N))\\
  &\qquad\leq C\,\frac{\log N}{\sqrt{N}}+\mathbb{E}\,\sum_{m\geq 0,\;q_{m}>0,\;q_{m+1}=0} \left(q_m+\frac{1}{N}\right).
  \end{split}
  \]
Hence, we deduce that
\[\begin{split}
  & \frac{\dd}{\dd t} \mathbb{E}\left[\sum_{n\geq 0} q_n\log q_n\right]\\
  &\quad\leq -\mathbb{E}\,\sum_{m\geq 0,q_m>0,q_{m+1}>0} (\bar r\,(q_m+1/N)-(q_{m+1}+1/N))\, \log\frac{\bar r(q_{m}+1/N)}{q_{m+1}+1/N} \\
  &\quad -2\,\mathbb{E}\,\sum_{m\geq 0,\;q_{m}=0,\;q_{m+1}>0}  \left(q_{m+1}+\frac{1}{N}\right)-2\,\mathbb{E}\,\sum_{m\geq 0,\;q_{m}>0,\;q_{m+1}=0} \left(q_m+\frac{1}{N}\right)+\frac{C\,\log N}{\sqrt{N}}.
  \end{split}
\]
Finally we note that if $a,\,b>0$ then
$(a-b)\,\log \frac{a}{b}\geq (a-b)\,L(a/b)$. If $a>0$ and $b=0$ then
$2\,a= (a-b)\,L(a/b)$, which allows us to conclude the proof.
\end{proof}
\subsection{Propagating exponential moments}\label{subsec:3.2}
Our next step is to obtain exponential moments on the $q_n$ in the spirit of Corollary~\ref{cor1}. The first steps consist in trying to obtain an equivalent of Proposition~\ref{prop1} to control the now random variable $\bar r(t)=\sum_{n\geq 1} q_n$.

\begin{lemma}
  For any smooth function $\phi:\;\R_+ \to\R$ and for any $k\geq 1$, we have that
  \[
  \begin{split}
    &\frac{\dd}{\dd t} \mathbb{E}\, \phi(q_k)= \\
&\quad   \mathbb{E}\,N\,\left[q_{k+1}\,\left(1-\frac{1}{N}-q_k\right)+q_{k-1}\,\left(\bar r-\frac{1}{N}\,\mathbbm{1}_{k>1}-q_k\right)\right]\,\left(\phi\left(q_k+\frac{1}{N}\right)-\phi(q_k)\right)\\
    &\quad+\mathbb{E}\,N\,q_{k}\,\left(1+\bar r-\frac{2}{N}-q_{k-1}-q_{k+1}\right)\, \left(\phi\left(q_k-\frac{1}{N}\right)-\phi(q_k)\right)\\
    &\quad+\mathbb{E}\,N\,q_{k-1}\,q_{k+1}\, \left(\phi\left(q_k+\frac{2}{N}\right)-2\,\phi\left(q_k+\frac{1}{N}\right) +\phi(q_k)\right)\\
    &\quad+\mathbb{E}\,N\,q_{k}\,\left(q_{k}-\frac{1}{N}\right)\, \left(\phi\left(q_k-\frac{2}{N}\right)-2\,\phi\left(q_k-\frac{1}{N}\right) +\phi(q_k)\right).
  \end{split}
  \]
  In the case $k=0$, we simply have that
  \begin{align*}
    \frac{\dd}{\dd t} \mathbb{E}\, \phi(q_0) &= \quad  \mathbb{E}\,N\,q_1\,\left(\bar r-\frac{1}{N}\right)\,\left(\phi\left(q_0+\frac{1}{N}\right)-\phi(q_0)\right)\\
    &\quad+ \mathbb{E}\,N\,q_0\,\left(\bar r-q_1\right)\,\left(\phi\left(q_0-\frac{1}{N}\right)-\phi(q_0)\right),
  \end{align*}
  and we recall that $\bar r$ is the random variable defined as $\bar r=\sum_{n\geq 1} q_n$.
  \label{phik}
  \end{lemma}
\begin{proof}
  This follows directly from Eq.~\eqref{eq:weak_formulation}. Note that for $\psi(\mathbf{q}) := \phi(q_k)$, for $\ell=k+1$ and $m\neq k-1,k$ or for $m=k-1$ and $\ell\neq k,k+1$
  \[
\left(\psi\left({\bf q} + \frac{1}{N}\left(\delta_{\ell-1} + \delta_{m+1} - \delta_\ell - \delta_m\right)\right) - \psi({\bf q})\right)=\phi\left(q_k+\frac{1}{N}\right)-\phi(q_k),
\]
while for $\ell=k$ and still $m\neq k-1,k$, or for $m=k$ and $\ell\neq k,k+1$
  \[
\left(\psi\left({\bf q} + \frac{1}{N}\left(\delta_{\ell-1} + \delta_{m+1} - \delta_\ell - \delta_m\right)\right) - \psi({\bf q})\right)=\phi\left(q_k-\frac{1}{N}\right)-\phi(q_k).
\]
Similarly for $\ell=k+1$ and $m=k-1$,
\[
\begin{split}
  &\left(\psi\left({\bf q} + \frac{1}{N}\left(\delta_{\ell-1} + \delta_{m+1} - \delta_\ell - \delta_m\right)\right) - \psi({\bf q})\right)=\phi\left(q_k+\frac{2}{N}\right)-\phi(q_k)\\
  &\qquad=2\,\left(\phi\left(q_k+\frac{1}{N}\right)-\phi(q_k)\right) +\phi\left(q_k+\frac{2}{N}\right)-2\,\phi\left(q_k+\frac{1}{N}\right)+\phi(q_k),
  \end{split}
  \]
  and for $\ell=k=m$,
  \[
\begin{split}
  &\left(\psi\left({\bf q} + \frac{1}{N}\left(\delta_{\ell-1} + \delta_{m+1} - \delta_\ell - \delta_m\right)\right) - \psi({\bf q})\right)=\phi\left(q_k-\frac{2}{N}\right)-\phi(q_k)\\
  &\qquad=2\,\left(\phi\left(q_k-\frac{1}{N}\right)-\phi(q_k)\right) +\phi\left(q_k-\frac{2}{N}\right)-2\,\phi\left(q_k-\frac{1}{N}\right)+\phi(q_k).
  \end{split}
  \]
The contribution is $0$ in all other cases, namely $\ell=k+1$ and $m=k$ or $\ell=k$ and $m=k-1$.

Hence
\[
\begin{split}
  &\mathbb{E}\,\sum_{\ell\geq 1,m\geq 0} N\,q_\ell\,\left(q_m-\frac{1}{N}\,\mathbbm{1}_{\ell=m}\right)\,\left(\psi\left({\bf q} + \frac{1}{N}\left(\delta_{\ell-1} + \delta_{m+1} - \delta_\ell - \delta_m\right)\right) - \psi({\bf q})\right)\\
  &\ =N\,\mathbb{E}\,\left(\phi\left(q_k+\frac{1}{N}\right)-\phi(q_k)\right)\,q_{k+1}\,\sum_{m\geq 0,\,m\neq k-1,\,k} \left(q_m-\frac{1}{N}\,\mathbbm{1}_{m = k+1}\right)\\
  &\quad +N\,\mathbb{E}\, \left(\phi\left(q_k+\frac{1}{N}\right)-\phi(q_k)\right)\,\sum_{\ell\geq 1,\,\ell\neq k,\,k+1} \left(q_\ell-\frac{1}{N}\,\mathbbm{1}_{\ell=k-1}\right)\,q_\ell\\
  &\quad +N\,\mathbb{E}\, \left(\phi\left(q_k-\frac{1}{N}\right)-\phi(q_k)\right)\\
  &\qquad\qquad \left(q_k\,\sum_{m\geq 0,\,m\neq k-1,\,k} \left(q_m-\frac{1}{N}\,\mathbbm{1}_{m=k}\right)+\sum_{\ell\geq 1,\,\ell\neq k,\,k+1} \left(q_k-\frac{1}{N}\,\mathbbm{1}_{\ell=k}\right)\,q_\ell\right)\\
  &\quad +N\,\mathbb{E}\,q_{k+1}\,q_{k-1}\, \left(\phi\left(q_k+\frac{2}{N}\right)-2\,\phi\left(q_k+\frac{1}{N}\right) +\phi(q_k) + 2\,\left(\phi\left(q_k+\frac{1}{N}\right)-\phi(q_k)\right) \right)\\
  &\quad +N\,\mathbb{E}\,q_{k}\,\left(q_{k}-\frac{1}{N}\right)\, \left(\phi\left(q_k-\frac{2}{N}\right)-2\,\phi\left(q_k-\frac{1}{N}\right) +\phi(q_k) + 2\,\left(\phi\left(q_k-\frac{1}{N}\right)-\phi(q_k)\right)\right),
\end{split}
\]
which gives the first result.

The case $k=0$ follows similar calculations except that we cannot have $\ell=k$ or $m=k-1$, leading to
\[
\begin{split}
  &\mathbb{E}\,\sum_{\ell\geq 1,m\geq 0} N\,q_\ell\,\left(q_m-\frac{1}{N}\,\mathbbm{1}_{\ell=m}\right)\,\left(\psi\left({\bf q} + \frac{1}{N}\left(\delta_{\ell-1} + \delta_{m+1} - \delta_\ell - \delta_m\right)\right) - \psi({\bf q})\right)\\
  &\ =N\,\mathbb{E}\,q_1\,\left(\phi\left(q_0+\frac{1}{N}\right) -\phi(q_0)\right)\,\sum_{m\geq 1} \left(q_m-\frac{1}{N}\,\mathbbm{1}_{m=1}\right)\\
  &\quad+N\,\mathbb{E}\,q_0\,\left(\phi\left(q_0-\frac{1}{N}\right) -\phi(q_0)\right)\,\sum_{\ell\geq 2} q_\ell,
  \end{split}
\]
which concludes the proof.
\end{proof}
This allows us to deduce the following
\begin{lemma}
  There exists a constant $\lambda>0$ and a constant $r_0<1$ such that
  \[
\sup_{t\geq 1}\,\mathbb{P}(\bar r(t)\geq r_0)\leq \exp\left(-\frac{N}{\lambda}\right).
  \]
  \label{1/Nr}
\end{lemma}
\begin{proof}
  We follow the main steps of the proof of Proposition~\ref{prop1}. First of all, we have $N_0$ and $C_\mu$ s.t. for any $t_0$, there exists a random $k_0\leq N_0$ with $q_{k_0}\geq 1/C_\mu$.

  This lets us define $N_0+1$ subsets of events $\omega_{k_0}$, $k_0=0,\ldots, N_0$, for which $q_{k_0}(t_0)\geq C_\mu$. We also note that the $\omega_{k_0}$ are defined by the state of the system at $t_0$. By the Markov property of the dynamics we can hence apply Lemma~\ref{phik} to the conditional expectation on each $\omega_{k_0}$. For simplicity we take $t_0=0$ here until the end of the proof.

We define $\phi_N(q)=\exp(N\,(1/(1+q)-1/(1+C_\mu)))$.
We first choose $k=k_0$ and apply Lemma~\ref{phik} to $\phi(t,q_k)=\phi_N(q_{k_0}+\tau(t-t_0))=\phi_N(q_{k_0}+\tau\,t)$. We observe that $\phi$ is decreasing in $q$ with
\[
\partial_q \phi(t,q_k)=-\frac{N}{(1+q_k+\tau\,t)^2}\,\exp(N\,(1/(1+q_k+\tau\,t)-1/(1+C_\mu))).
\]
This implies that $\phi(t,q_k+1/N)-\phi(t,q_k)\leq 0$ while
\[
\phi(t,q_k-1/N)-\phi(t,q_k)\leq \frac{1}{N}\,|\partial_q \phi(t,q_k-1/N)|\leq -\frac{C}{N}\,\partial_q \phi(t,q_k),
\]
for some constant $C$.

On the other hand $\phi$ is convex in $q$ with
\[
\begin{split}
  \partial_q^2 \phi(t,q_k)&=\left(\frac{N^2}{(1+q_k+\tau\,t)^4}+\frac{2N}{(1+q_k+\tau\,t)^3}\right)\,\exp(N\,(1/(1+q_k+\tau\,t)-1/(1+C_\mu)))\\
  &\leq C\,N\,|\partial_q \phi(t,q_k)|,
\end{split}
\]
again for some constant $C$. Therefore we also obtain that
\[
\begin{split}
&\phi(t,q_k+2/N)-2\,\phi(t,q_k+1/N)+\phi(t,q_k)\leq \frac{C}{N}\,\left|\partial_q \phi(t,q_k)\right|,\\
&\phi(t,q_k-2/N)-2\,\phi(t,q_k-1/N)+\phi(t,q_k)\leq \frac{C}{N}\,\left|\partial_q \phi(t,q_k)\right|.
\end{split}
\]
Lemma~\ref{phik} then implies that for some $C>0$ and for any $t\geq t_0=0$,
  \[
\mathbb{E}\left(\phi(t,q_{k_0})\,|\;\omega_{k_0}\right)\leq \phi(t_0,C_\mu)+\mathbb{E}\int_{t_0}^t \partial_q \phi(s,q_{k_0})\,(\tau-C)\,ds,
\]
since $\partial_t \phi(t,q)=\tau\,\partial_q \phi(t,q)$. Choosing $\tau\geq C$ and as $\partial_q \phi(t,q) \leq 0$ yields that
  \[
  \mathbb{E}\left(\phi(t,q_{k_0})\,|\;\omega_{k_0}\right)\leq 1.
  \]
  For any $t\leq t_0+t_1$ with $\tau t_1\leq C_\mu/2$, we hence have a constant $C_0$ s.t.
  \begin{equation}
\sup_{t_0\leq t\leq t_0+t_1}\mathbb{E} \exp\left(\frac{N}{C_0}\,\left(\frac{1}{1+q_{k_0}}-\frac{1}{1+C_0}\right)\right)\leq C_0.\label{expk0}
  \end{equation}
  We now turn to $k=k_0-1$ and choose $\phi_N(q)=\exp(N\,(1/(1+q)-1)/\kappa)$ with hence $\phi_N(0)=1$. We now define $\phi_1(t,q)=\phi_N(q-\tau\,(t-t_0))=\phi_N(q-\tau\,t)$ (observe that it is $q-\tau\,t$ instead $q+\tau\,t$ as for $k=k_0$).

As long as $\tau\,(t-t_0)\leq 1/2$,  we have similar properties on $\phi_1$ as before. But need to be more precise on the dissipation term with
  \[
\phi_1(t,q_k+1/N)-\phi_1(t,q_k)\leq \frac{1}{C\,N}\,\partial_q \phi_1(t,q_k).
\]
Furthermore we note that the presence of $\kappa$ leads to an improved estimate on the second derivatives with
\[
\partial_q^2 \phi_1(t,q)\leq N\,(C/\kappa+C/N)\,|\partial_q \phi_1(t,q)|.
\]
  Lemma~\ref{phik} now leads to
  \[\begin{split}
  &\mathbb{E}\left(\phi_1(t,q_{k_0-1})\,|\;\omega_{k_0}\right)\leq 1+\mathbb{E}\int_{t_0}^t \partial_q \phi_1(s,q_{k_0-1})\\
  &\qquad\left(\frac{q_{k_0}}{C}\,\left(1-\frac{1}{N}-q_{k_0-1}\right)-\tau -C\,q_{k_0-1}-C/\kappa-C/N\right)\,ds.
\end{split}
\]
First of all, provided $\kappa$ is large enough, we have that
\[
|\partial_q \phi_1(s,q_{k_0-1})|\leq |\partial_q \phi_1(s,0)|\leq C\,N\,\expo^{N/\kappa}\leq \exp\left(\frac{N}{C_0}\,\left(\frac{1}{1+C_0/2}-\frac{1}{1+C_0}\right)\right).
\]
Hence if $q_{k_0}\leq C_0/2$, we immediately have that
\[
|\partial_q \phi_1(s,q_{k_0-1})|\leq |\partial_q \phi_1(s,0)|\leq C\,N\,\expo^{N/\kappa}\leq \exp\left(\frac{N}{C_0}\,\left(\frac{1}{1+q_{k_0}}-\frac{1}{1+C_0}\right)\right).
\]
If $q_{k_0}\geq C_0/2$, by choosing $\kappa$ large enough and $\tau$ small enough, we note that
\[
\frac{q_{k_0}}{C}\,\left(1-\frac{1}{N}-q_{k_0-1}\right)-\tau -C\,q_{k_0-1}-C/\kappa-C/N\geq 0,
\]
provided that $q_{k_0-1}\leq q_{k_0}/\tilde{C}$ with $\tilde{C}$ being large enough.

On the other hand, if $q_{k_0-1}\geq q_{k_0}/\tilde{C}$ and for $\tau\,(t-t_0)\leq 1/2$ then
\[
\begin{split}
  |\partial_q \phi_1(t,q_{k_0-1})|\leq  |\partial_q \phi_1(t,q_{k_0}/\tilde{C})|&\leq C\,N\,\exp\left(C\,\frac{N}{\kappa}\,\left(\frac{1}{C/2+q_{k_0}}-1/C\right)\right)\\
  &\leq \exp\left(\frac{N}{C_0}\,\left(\frac{1}{1+q_{k_0}}-\frac{1}{C_0+1}\right)  \right),
\end{split}
\]
provided that $C\geq 2$, $C\leq C_0+1$ and $\kappa$ was chosen large enough again.

Thus in all cases, we may write that
\[
|\partial_q \phi_1(t,q_{k_0-1})|\leq \exp\left(\frac{N}{C_0}\,\left(\frac{1}{1+q_{k_0}}-\frac{1}{C_0+1}\right)  \right).
\]
Therefore,
  \[\begin{split}
  \mathbb{E}\left(\phi_1(t,q_{k_0-1})\,|\;\omega_{k_0}\right)&\leq 1+\mathbb{E}\int_{t_0}^t \exp\left(\frac{N}{C_0}\,\left(\frac{1}{1+q_{k_0}}-\frac{1}{C_0+1}\right)  \right)\,ds\\
  &\leq 1+C_0\,(t-t_0).
\end{split}
\]
For any $\delta>0$, this implies that there exists a constant $C_1$ s.t. for all $t_0+\delta/N_0\leq t\leq t_0+t_1$
\begin{equation}
\sup_{t_0+\delta/N_0\leq t\leq t_0+t_1}\,\mathbb{E}\exp\left(\frac{N}{C_1}\,\left(\frac{1}{1+q_{k_0-1}} -\frac{1}{C_1+1}\right)  \right)\leq C_1.\label{expk-1}
  \end{equation}
By using \eqref{expk-1}, we may now repeat the argument starting from $k_0-1$ and $t_0+\delta/N_0$ instead of $t_0$, until we reach $k=0$ and $t_0+\delta$. Since we can choose any $t_0\geq 0$ and since $q_0=1-\bar r$, this yields the advertised result by Markov's inequality.
\end{proof}
The next step is to extend the result of Corollary~\ref{cor1} under the random setting. This uses in a very strong manner the large deviation bound obtained in Lemma~\ref{1/Nr}.
\begin{lemma}
There exists a constant $K>1$ independent of $N$ such that
  \[
\sup_{t>0} \mathbb{E}\sum_{n\geq 0} K^{n} q_n\leq 2.
\]\label{Kn}
\end{lemma}
\begin{proof}
  We use again \eqref{eq:weak_formulation} with $\psi({\bf q})=\sum_n K^n\,q_n$. Since  $\psi$ is linear in ${\bf q}$, we simply have that
  \[
  \begin{split}
   & \psi\left({\bf q} + \frac{1}{N}\left(\delta_{\ell-1} + \delta_{m+1} - \delta_\ell - \delta_m\right)\right) - \psi({\bf q})=\psi\left(\frac{1}{N}\left(\delta_{\ell-1} + \delta_{m+1} - \delta_\ell - \delta_m\right)\right)\\
   &\quad=\frac{K^{\ell-1}+K^{m+1}-K^{\ell}-K^{m}}{N}=(K-1)\,\frac{K^m-K^{\ell-1}}{N}.
  \end{split}
  \]

Hence we first obtain that
\[\frac{\dd}{\dd t}\,\mathbb{E}\sum_{n\geq 0} K^n q_n =(K-1)\,\mathbb{E}\,\left[\sum_{\ell \geq 1, m\geq 0} q_\ell\,\left(q_m-\frac{1}{N}\,\mathbbm{1}\{\ell = m\}\right) (K^m-K^{\ell-1}).\right]\]
Recalling our definition $\bar{r} =\sum_{\ell\geq 1} q_\ell=1-q_0$, this leads to computations very similar to the proof of Corollary~\ref{cor1}
\[
\begin{split}
  \frac{\dd}{\dd t}\,\mathbb{E}\sum_{n\geq 0} K^n q_n&\leq (K-1)\,\mathbb{E}\,\left[ \bar r\,\sum_{m\geq 0} q_m\,K^m\right]-\frac{K-1}{K}\,\mathbb{E}\,\sum_{\ell\geq 1} q_\ell\,K^\ell\\
  &=(1-1/K)\,\mathbb{E}\,q_0 + (K-1)\,\mathbb{E}\,\left[(\bar r-1/K)\,\sum_{\ell\geq 0}q_\ell\,K^\ell\right].
\end{split}
\]
Imposing that $1/K> r_0$ or $K<1/r_0$ and separate the case $\bar r\leq r_0$ and $\bar r\geq r_0$, together with the elementary observation that $\sum_\ell q_\ell\,K^\ell\leq K^{N\,\mu}$ lead us to
\[
\mathbb{E}\,\left[(\bar r-1/K)\,\sum_{\ell\geq 0}q_\ell\,K^\ell\right]\leq (r_0-1/K)\,\mathbb{E}\,\left[\sum_{\ell\geq 0}q_\ell\,K^\ell\right]+K^{N\,\mu}\,\mathbb{P}(\bar r\geq r_0).
\]
Applying Lemma~\ref{1/Nr} yields
\[
\mathbb{E}\,\left[(\bar r-1/K)\,\sum_{\ell\geq 0}q_\ell\,K^\ell\right]\leq (r_0-1/K)\,\mathbb{E}\,\left[\sum_{\ell\geq 0}q_\ell\,K^\ell\right]+K^{N\,\mu}\,\expo^{-N/\lambda}.
\]
Enforcing that $K^{N\,\mu}\,\expo^{-N/\lambda}\leq 1$ or $\log K\leq 1/(\lambda\,\mu)$, we obtain
\[
\begin{split}
  \frac{\dd}{\dd t}\,\mathbb{E}\sum_{n\geq 0} K^n q_n&\leq
  2 + (K-1)\,(r_0-1/K)\mathbb{E}\,\left[\sum_{\ell\geq 0}q_\ell\,K^\ell\right].
\end{split}
\]
It is then easy to conclude by Gronwall's inequality that $\mathbb{E}\,\left[\sum_{\ell\geq 0}q_\ell\,K^\ell\right]$ is uniformly bounded in time. Lowering further $K$, it is also always possible to ensure that this quantity less than $2$ by interpolation with $K=1$.
\end{proof}

\begin{remark}
We do emphasize here that Lemma~\ref{1/Nr} was critical as it allows us to replace the random variable $\bar{r} -1/K$, which would have blocked any attempt to use Gronwall-type inequalities, by the deterministic $r_0-1/K$. The exponential tail in Lemma~\ref{1/Nr} is in particular essential to compensate the effect of the term $K^{N\,\mu}$.
\end{remark}
\subsection{The precise $\log$-Sobolev inequality}\label{subsec:3.3}
After we obtained a good control on the exponential moments, we need to make the quantitative logarithmic Sobolev inequality in Theorem~\ref{thm1} as explicit as possible in terms of $K$ and $\bar r$ due to the large deviations considerations. For this purpose, we prove a pointwise version as follows:
\begin{proposition}
  For any random probability mass function $\{p_n\}_{n \geq 0}$ satisfying the following assumptions
  \[
\sum_{n\geq 0} n\,p_n=\tilde \mu,\quad \sum_{n\geq 1} p_n=\bar r<1,\quad \mathbb{E}\,\sum_{n\geq 0} K^n\,p_n\leq 2,
\]
for some $K>1$, we have that
\[
\mathbb{E}\,(1-\bar r)\,\sum_n p_n\,\log\frac{p_n}{p_n^*}\leq \frac{C}{K-1}\,\sqrt{\tilde D}\,|\log \tilde D|+\frac{C}{K-1}\,\mathbb{E}\,|\mu-\tilde\mu|\,|\log(\mathbb{E}\,|\mu-\tilde\mu|)|,
\]
in which the modified dissipation term $\tilde D$ is defined via
\[
\tilde D =\mathbb{E}\sum_{n\geq 0} (\bar r\, p_n-p_{n+1})\,L\left(\frac{\bar r\,p_n}{p_{n+1}}\right),
\]
where we recall that the function $L$ is defined by $L(x)=\log x$ if $-2\leq\log x\leq 2$, $L(x)=-2$ if $\log x\leq -2$ and $L(x)=2$ if $\log x\geq 2$.
\label{preciselogsobolev}
\end{proposition}
Note that in the proposition, we may have that $\sum n\,p_n\neq \sum n\,p_n^*$ as $p_n^*$ is defined through $r^*$ and $\mu$ while $\sum n\,p_n=\tilde \mu$.
\begin{proof}
As before, we write that
\begin{equation}\label{eq:decompose_H_random}
\HH = \mathbb{E}\,(1-\bar r)\,\sum_{n \geq 0} \left(p_n\,\log \frac{p_n}{p^*_n} + p^*_n - p_n\right) := \HH_1 + \HH_2,
\end{equation}
with
\begin{equation}\label{eq:H1_random}
\HH_1 := \mathbb{E}\,(1-\bar r)\,\sum_{n \colon p_n \geq Mp^*_n} \left(p_n\,\log \frac{p_n}{p^*_n} + p^*_n - p_n\right)
\end{equation}
and
\begin{equation}\label{eq:H2_random}
\HH_2 := \mathbb{E}\,(1-\bar r)\,\sum_{n \colon p_n < Mp^*_n} \left(p_n\,\log \frac{p_n}{p^*_n} + p^*_n - p_n\right)
\end{equation}
for some (potentially) large $M > 0$ to be determined. Let $\theta > 0$ be a small enough exponent so that $(r^*)^{2\theta} > 1/K^2$. The estimate of $\HH_1$ can be carried out as follows for constants $C_1$ and $C_2$,
\begin{align*}
\HH_1 &\leq \mathbb{E}\,\sum_{n \colon p_n \geq Mp^*_n} \left(p_n\,\log \frac{p_n}{p^*_n} + p^*_n - p_n\right) \\
&\leq C_1\,\mathbb{E}\,\sum_{n \colon p_n \geq Mp^*_n} p_n(n + C_2) \\
&\leq \frac{C_1}{M^\theta}\mathbb{E}\,\sum_{n \colon p_n \geq Mp^*_n} K^{\frac{n}{2}}p_n \left(\frac{p_n}{p^*_n}\right)^\theta \frac{n+C_2}{K^{\frac{n}{2}}}\\
&\leq \frac{C_3}{M^\theta}\left(\mathbb{E}\,\sum_{n\geq 0} K^n p^{1+2\theta}_n \right)^{\frac 12}\left(\mathbb{E}\,\sum_{n\geq 0} p_n\,\frac{(n+C_2)^2}{\left(K(r^*)^{2\theta}\right)^n}\right)^{\frac 12} \leq C/M^\theta\,\left(\frac{1}{\log \left(K^2\,(r^*)^{2\theta}\right)}\right)^{\frac 12},
\end{align*}
in which $C_1, C_2, C_3 > 0$ and $C >0$ are fixed constants depending only on $r^*$ and $K$.

Next, notice that for each $n \in \mathbb N$, we have
\begin{equation}\label{eq:identity_restate}
p_n - p^*_n = \sum_{0\leq k\leq n-1} \bar{r}^{n-k-1}\left(p_{k+1} - \bar{r}p_k\right) + \left(\bar{r}^n p_0 - (r^*)^n p^*_0\right).
\end{equation}
Then for some constant $C > 0$ whose value might change from line to line, but again depending only on $\mu$
\begin{align*}
\sum_{n\geq 0} \frac{|p_n - p^*_n|^2}{p_n + p^*_n} &\leq \sum_{n\geq 0} |p_n - p^*_n| \\
&\leq \sum_{k \geq 0} |p_{k+1}-\bar{r}p_k|\sum_{n\geq k+1} \bar{r}^{n-k-1} + \sum_{n\geq 0} \left|\bar{r}^n p_0 - (r^*)^n p^*_0\right| \\
&\leq \frac{C}{1-\bar r}\,|p_0 - p^*_0| + \frac{C}{1-\bar r}\,\left(\sum_{k\geq 0} \frac{|p_{k+1}-\bar{r}p_k|^2}{p_{k+1}+\bar{r}p_k}\right)^{\frac 12}.
\end{align*}
We note that the elementary fact~\eqref{eq:elementary} also applies to the function~$L$, i.e.,
\[
\frac{(a-b)^2}{a+b}\leq (a-b)\,L\left(\frac{a}{b}\right).
\]
Indeed, if $\expo^{-2}\leq a/b\leq \expo^2$ then this is the same inequality as $L(a/b)=\log a/b$. If $a/b\geq \expo^2$ for example then $(a-b)^2/(a+b)\leq a$ while $(a-b)\,L(a/b)\geq \frac{a}{2}\cdot 2$.

We can hence also use Lemma~\ref{lem1} with $\tilde D$ as the proof of Lemma~\ref{lem1} relies on the same inequalities. Because of the possibility that $\mu\neq\tilde\mu$, we obtain
\[
\mathbb{E}\,|p_0-p_0^*|\leq C\,\mathbb{E}\,|\mu-\tilde\mu|+C\,\sqrt{\tilde D}.
\]
This implies that
\begin{align*}
\mathbb{E}\,(1-\bar r)\,\sum_{n\geq 0} \frac{|p_n - p^*_n|^2}{p_n + p^*_n}&\leq C\,\mathbb{E}\,|\mu-\tilde\mu|+C\,\sqrt{\tilde D},
\end{align*}
in which the constant $C$ is independent of $\bar r$.

Here we emphasize that the estimates presented here are somewhat worse than its counterpart in the deterministic setting (shown in the proof of Theorem \ref{thm1}), as our control on $\bar{r}$ (which is now a random variable) might be poor a priori and $\bar{r}$ might be very close to $1$ for some realizations of the stochastic $N$-agent system, which means that whenever we have a series like $\sum \bar{r}^{\,n}$, we need to bear with an extra $1/(1-\bar r)$.

To control $\HH_2$, we use again the observation that for any $a,b \in \mathbb{R}_+$, if $a \leq Mb$ for some $M > 0$, then $(a-b)\log(a/b) + b - a \leq C\,\left|\log M\right|\,\frac{|a-b|^2}{a+b}$ for some constant $C > 0$ (independent of $a,b$ and $M$). Thus, we deduce that
\begin{equation}\label{eq:bound_H2_random}
\HH_2 \leq C\,\left|\log M\right|\,\mathbb{E}\,(1-\bar r)\,\sum_{n\geq 0} \frac{|p_n - p^*_n|^2}{p_n + p^*_n} \leq C\,(\sqrt{\tilde D}+\mathbb{E}\,|\mu-\tilde\mu|)\,\left|\log M\right|.
\end{equation}
Putting the estimates on $\HH_1$ and $\HH_2$ together lead us to
\begin{equation}\label{eq:bound_H_random}
\HH = \HH_1 + \HH_2 \leq \frac{C}{M^\theta}\,\left(\frac{1}{\log \left(K^2\,(r^*)^{2\theta}\right)}\right)^{\frac 12} + C\,(\sqrt{\tilde D}+\mathbb{E}\,|\mu-\tilde\mu|)\,\left|\log M\right|,
\end{equation}
for some constant $C > 0$.

As before, we may first take $\theta = (K-1)/C$ for sufficiently large $C > 0$ such that (since $K$ cannot be too large)
\[
\log \frac{K^2}{(r^*)^{2\,\theta}} \geq \log K\geq \frac{K-1}{C}.
\]
The bound on $\HH$ follows by taking $M = \left(1/(\mathbb{E}\,|\mu-\tilde\mu|+\tilde D)\right)^{1/\theta}$ which implies that
\[
|\log M|\leq \frac{1}{\theta}\,\min(|\log\mathbb{E}\,|\mu-\tilde\mu||,\ |\log \tilde D|)\leq \frac{C}{K-1}\,\min(|\log\mathbb{E}\,|\mu-\tilde\mu||,\ |\log \tilde D|),
\]
for some $C$ depending only on $\mu$ or $r^*$.
\end{proof}
\subsection{Uniform propagation of chaos}\label{subsec:3.4}
We now assemble the previous estimates to show that the random histogram $\{q_n\}_{n \geq 0}$ still converge to the deterministic equilibrium $\{p_n^*\}_{n \geq 0}$ up to a small polynomial correction in $N$.
\begin{theorem}\label{thm3}
For all $t > 0$, we have that
\[
\mathbb{E}\,\sum_{n\geq 0} q_n(t)\,\log \frac{q_n(t)}{p_n^*}  \leq C\,\frac{\log N}{N^{1/4}}+\frac{C}{1+t}.
\]
\end{theorem}
\begin{proof}
We introduce
\[H(t)= \mathbb{E}\,\sum_{n\geq 0} q_n(t)\,\log \frac{q_n(t)}{p_n^*}\]
and
\[\tilde q_n=\left(q_n+\frac{1}{N}\right)\,\mathbbm{1}\{q_n>0\},\quad \tilde q_n=0\ \mbox{if}\ q_n=0. \]
We notice that a slight variant of Lemma~\ref{agentsentropy} enables us to deduce that
  \[
  \begin{split}
    \frac{\dd}{\dd t} H(t)\leq &-\mathbb{E}\sum_{n\geq 0} (\bar r\,\tilde q_n-\tilde q_{n+1})\,L\left(\frac{\bar r\,\tilde q_n}{\tilde q_{n+1}}\right)+C\,\frac{\log N}{\sqrt{N}}.
  \end{split}
  \]
  We now need to normalize the $\tilde q_n$ so that they define a probability distribution, hence we introduce
  \[
\bar q_n=\frac{\tilde q_n}{\bar m},\quad \bar m=\sum_{n\geq 0} \tilde q_n.
\]
We note that a variant of \eqref{inter2} shows that
\[
|\bar m-1|\leq \frac{C}{\sqrt{N}},
\]
almost surely, for a deterministic $C$ depending only on $\mu$.

We can easily renormalize and obtain
 \[
  \begin{split}
    \frac{\dd}{\dd t} H(t)&\leq -\mathbb{E}\,\bar m\,\sum_{n\geq 0} (\bar r\,\bar q_n-\bar q_{n+1})\,L\left(\frac{\bar r\,\bar q_n}{\bar q_{n+1}}\right)+C\,\frac{\log N}{\sqrt{N}}\\
    &\leq -\frac{1}{2}\,\mathbb{E}\,\sum_{n\geq 0} (\bar r\,\bar q_n-\bar q_{n+1})\,L\left(\frac{\bar r\,\bar q_n}{\bar q_{n+1}}\right)+C\,\frac{\log N}{\sqrt{N}},
\end{split}
  \]
  since $L$ is non-decreasing.

  Denote $\tilde \mu=\sum_{n\geq 0}n\,\bar q_n$. Observe that 
  \[
\tilde\mu=\frac{1}{\bar m}\,\left(\sum_n n\,q_n+\sum_{n,\,q_n>0} \frac{n}{N}\right)=\frac{1}{\bar m}\,\left(\mu+\sum_{n,\,q_n>0} \frac{n}{N}\right).
\]
With a calculation similar to \eqref{inter2}, we find that
\[
\mathbb{E}\,\sum_{n,\,q_n>0} \frac{n}{N}\leq \mathbb{E}\,\sum_{n\leq N^{1/4}} \frac{n}{N}+\mathbb{E}\sum_{n\geq N^{1/4}} n\,q_n\leq \frac{1}{\sqrt{N}}+\frac{C\,N^{1/4}}{K^{N^{1/4}}}\,\mathbb{E}\sum_{n\geq 0} K^n\,q_n\leq \frac{C}{\sqrt{N}},
\]
by Lemma~\ref{Kn}.

This implies that
\[
\mathbb{E}\,|\mu-\tilde\mu|\leq \frac{C}{\sqrt{N}}.
\]
We may now apply Proposition~\ref{preciselogsobolev} combined again with the exponential moment estimate from Lemma~\ref{Kn} to find that
  \[
\frac{\dd}{\dd t} H(t)\leq -\frac{1}{C}\,\frac{\tilde H^2}{(\log \tilde H)^2}+C\,\frac{\log N}{\sqrt{N}}
\]
for
\[
\tilde H=\mathbb{E}\,(1-\bar r)\,\sum_n \bar q_n\,\log\frac{\bar q_n}{p_n^*}-\frac{C}{\sqrt{N}},
\]
and as long as $\tilde H\geq 0$.

Remark that
\[
\begin{split}
\mathbb{E}\,(1-\bar r)\,\sum_n \bar q_n\,\log\frac{\bar q_n}{p_n^*}&\geq (1-r_0)\,
\mathbb{E}\,\sum_n \bar q_n\,\log\frac{\bar q_n}{p_n^*}-\mathbb{E}\,\tilde\mu\,\mathbbm{1}_{\bar r>r_0}\\
&\geq (1-r_0)\,
\mathbb{E}\,\sum_n \bar q_n\,\log\frac{\bar q_n}{p_n^*}-C\,\expo^{-N/\lambda},
\end{split}
\]
by Lemma~\ref{1/Nr}.

Similar estimates yield that
\[
\mathbb{E}\,\sum_n \bar q_n\,\log\frac{\bar q_n}{p_n^*}\geq \mathbb{E}\,\sum_n q_n\,\log\frac{q_n}{p_n^*}-\frac{C}{\sqrt{N}}.
\]
Finally we deduce that
  \[
\frac{\dd}{\dd t} H(t)\leq -\frac{1}{C}\,\frac{(H-C/\sqrt{N})^2}{(\log (H-C/\sqrt{N}))^2}+C\,\frac{\log N}{\sqrt{N}},
\]
from which we conclude that
\[H(t)\leq C\,\frac{\log N}{N^{1/4}}+\frac{C}{1+t}.\]
Thus the proof is completed.
\end{proof}
Now we ready to prove the uniform (in time) propagation of chaos results (i.e, Theorem \ref{thm4}), which is restated here for the reader's convenience.

\addtocounter{theorem}{-5}%
\begingroup
\begin{theorem}
For all $t > 0$ and all $N \geq 1$, we have that
\begin{equation}\label{eq:L1chaos_restate}
\mathbb{E}\,\|{\bf q}(t) - {\bf p}(t)\|^2_{\ell^1} \leq \frac{C}{1+\log N},
\end{equation}
in which ${\bf p}(t)=\{p_n(t)\}_{n\geq 0}$ is the unique solution of ODE system \eqref{eq:law_limit}. Furthermore, we have the entropic uniform in time propagation of chaos, i.e.,
\begin{equation}\label{eq:entropic_chaos_restate}
\mathbb{E}\,\sum_n q_n(t)\,\log\frac{q_n(t)}{p_n(t)}\leq C\,\frac{\log \log N}{1+\log N}.
\end{equation}
\end{theorem}
\endgroup

\begin{proof}
First, we note that a finite time propagation of chaos proved for instance in \cite{cao_derivation_2021} and \cite{merle_cutoff_2019} ensures that
\begin{equation}\label{eq:finite-time-L1-chaos}
\mathbb{E}\,\|{\bf q}(t) - {\bf p}(t)\|^2_{\ell^1} \leq C\,\frac{\expo^{C\,t}}{N} \quad \forall\, t \in (0,\infty)
\end{equation}
for some $C > 0$. By combining Theorem \ref{thm2}, Theorem \ref{thm3} and the celebrated Csisz\'{a}r-Kullback-Pinsker inequality, we also have
\begin{equation}\label{eq:triangle_inequality}
\mathbb{E}\,\|{\bf q}(t) - {\bf p}(t)\|^2_{\ell^1} \leq C\,\expo^{-C\,\sqrt{t}} + C\,\frac{\log N}{N^{1/4}}+\frac{C}{1+t} \quad \forall\, t > 0.
\end{equation}
Now the advertised bound \eqref{eq:L1chaos} follows by taking $T = \frac{\log N}{2\,C}$ and observing that
\begin{equation}
\mathbb{E}\,\|{\bf q}(t) - {\bf p}(t)\|^2_{\ell^1} \leq
    \begin{cases}
        \frac{C}{1+\log N} & \text{if } t \geq T,\\
        \frac{C}{\sqrt{N}} & \text{if } t \leq T.
    \end{cases}
\end{equation}
A straightforward interpolation argument then allows us to derive the relative entropy estimates as well. For example, one may write
\[
\begin{split}
  \sum_n q_n(t)\,\log\frac{q_n(t)}{p_n(t)}&=\sum_n \left(q_n(t)\,\log\frac{q_n(t)}{p_n(t)}+p_n(t)-q_n(t)\right)\\
  &\leq \sum_{n,\,M^{-1}\,p_n\leq q_n\leq M\,p_n} \left(q_n(t)\,\log\frac{q_n(t)}{p_n(t)}+p_n(t)-q_n(t)\right)\\
  &+\sum_{n,q_n>M\,p_n} q_n\, (1+n)+\frac{\log M}{M}\,\sum_{n} p_n.
\end{split}
\]
We have that
\[
\begin{split}
 & \mathbb{E}\,\sum_{n,\,M^{-1}\,p_n\leq q_n\leq M\,p_n} \left(q_n(t)\,\log\frac{q_n(t)}{p_n(t)}+p_n(t)-q_n(t)\right)\leq \log M\,\mathbb{E}\,\sum_{n} |p_n-q_n|^2\\
  &\qquad\leq 2\,\log M\,\mathbb{E}\,\sum_{n} |p_n-q_n|\leq \frac{C\,\log M}{1+\log N}.
\end{split}
\]
On the other hand,
\[
\sum_{n,q_n>M\,p_n} q_n\, (1+n)\leq \frac{1}{M}\,\sum_n p_n\,(1+n)\leq \frac{C}{M}.
\]
Hence
\[
\mathbb{E}\,\sum_n q_n(t)\,\log\frac{q_n(t)}{p_n(t)}\leq \frac{C\,\log M}{1+\log N}+\frac{C\,\log M}{M},
\]
which gives rise to \eqref{eq:entropic_chaos} by optimizing in the choice of $M$.
\end{proof}

\section{Conclusion}
\setcounter{equation}{0}

In this manuscript, an agent based model (known as the unbiased exchange model in \cite{cao_derivation_2021} and the one-coin model in \cite{lanchier_rigorous_2017}) for wealth (re-)distribution introduced in the seminar work \cite{dragulescu_statistical_2000} is studied. We rigorously proved a uniform in time propagation of chaos result for this model which is not available in the literature prior to the present work, based on a detailed analysis of dissipation of the relative entropy at the level of the $N$-agent system. Meanwhile, precise entropy-entropy dissipation arguments have also been carried out at the level of the limiting ODE system, which lead us to a quantitative almost-exponential rate of convergence toward a geometric equilibrium distribution of money. It would also be interesting to examine variants of this model, for instance the so-called poor-biased and rich-biased exchange model investigated in \cite{cao_derivation_2021}, where agents are picked to give a dollar at a rate based on the amount of dollars he/she has. One possible follow-up work would be to establish a uniform in time propagation of chaos result for these closely related models. As of now, one exciting direction of research for econophysics models involves the inclusion of a (central) bank and hence the possibility of agents being in debt (see for instance \cite{lanchier_rigorous_2018-1}). As the inclusion of bank and debt will also make the econophysics models more realistic from a modeling perspective, we plan to extend the framework and analysis of the present work to other econophysics models with bank and debt in our future work.

\end{document}